\documentclass[oneside,reqno]{amsart}
\usepackage{amsfonts,amsmath,latexsym,verbatim,amscd,mathrsfs,color,array}
\usepackage[colorlinks=true]{hyperref}

\usepackage{amsmath,amssymb,amsthm,amsfonts,graphicx,color}
\usepackage{amssymb}
\usepackage{pdfsync}
\usepackage{epstopdf}

\newcommand{\ass}{\quad\mbox{as}\quad}

\newcommand{\inn}{{\quad\hbox{in } }}
\newcommand{\onn}{{\quad\hbox{on } }}
\newcommand{\ttt}{\tilde }

\newcommand{\nn}{ {\nabla}  }

\newcommand{\pp}{ {\partial} }

\newcommand{\C}{{\mathbb C}}
\newcommand{\R} {\mathbb R}

\newcommand{\cuad}{{\sqcap\kern-.68em\sqcup}}

\newcommand{\DD}{{\mathcal D}}

\newcommand{\foral}{\quad\mbox{for all}\quad}
\newcommand{\ve}{\varepsilon}

\newcommand{\be}{\begin{equation}}
\newcommand{\ee}{\end{equation}}

\newcommand{\la}{\lambda}

\newcommand{\equ}[1]{(\ref{#1})}

\newtheorem{lemma}{Lemma}[section]
\newtheorem{prop}{Proposition}[section]
\newtheorem{theorem}{Theorem}

\newtheorem{remark}{Remark}[section]
\newcommand{\bremark}{\begin{remark} \em}
\newcommand{\eremark}{\end{remark} }

\numberwithin{equation}{section}

\begin{document}

\title[Type II blow up for the critical heat equation]{Geometry driven Type II higher dimensional blow-up for the critical heat equation}

\author[M. del Pino]{Manuel del Pino}
\address{\noindent   Department of Mathematical Sciences University of Bath,
Bath BA2 7AY, United Kingdom }
\email{m.delpino@bath.ac.uk}

\author[M. del Pino]{Monica Musso}
\address{\noindent   Department of Mathematical Sciences University of Bath,
Bath BA2 7AY, United Kingdom }
\email{m.musso@bath.ac.uk}

\author[J. Wei]{Juncheng Wei}
\address{\noindent  Department of Mathematics University of British Columbia, Vancouver, BC V6T 1Z2, Canada
}  \email{jcwei@math.ubc.ca}

\begin{abstract}
We consider the problem
\begin{align*}
v_t   & = \Delta v+ |v|^{p-1}v   \inn\ \Omega\times (0, T),  \\
v &  =0 \onn  \partial \Omega\times (0, T ) ,  \\
v&  >0\inn\  \Omega\times (0, T) .
\end{align*}
In a domain $\Omega\subset {\mathbb R  } ^d$, $d\ge 7$ enjoying special symmetries, we find the first example of a solution with type II blow-up for a power $p$ less than the Joseph-Lundgren exponent
$$p_{JL}(d)=\begin{cases} \infty, & \text{if $3\le d\le 10$},\\ 1+{4\over d-4-2\,\sqrt{d-1}}, & \text{if $d\ge11$}. \end{cases} $$
No type II radial blow-up is present for $p< p_{JL}(d)$.
 We take $p=\frac{d+1}{d-3}$, the Sobolev critical exponent in one dimension less. The solution blows up
on circle contained in a negatively curved part of the boundary in the form of a sharply scaled Aubin-Talenti bubble, approaching its energy density
a Dirac measure for the curve. This is a completely new phenomenon for a diffusion setting.
\end{abstract}

\maketitle

\section{Introduction}

Perhaps the most studied model of singularity formation or blow-up in nonlinear parabolic problems is the
semilinear heat equation

\begin{align}\label{P1}
v_t   & = \Delta v+ |v|^{p-1}v   \inn\ \Omega\times (0, T),  \\
%u   &\color{blue>0 \inn \Omega\times (0, \infty)\nonumber \\
v &  =0 \onn  \partial \Omega\times (0, T ) ,  \nonumber\\
v&  >0\inn\  \Omega\times (0, T) .\nonumber
\end{align}
where $p>1$. Here  $\Omega$ be a smooth domain in $ \R^d$ (or entire space) and $0<T\le +\infty$. A smooth solution $u(x,t)$ of Problem \equ{P1}
is said to blow-up at time $T$ if
$$
\lim_{t\to T} \|v(\cdot ,t)\|_{L^\infty(\Omega)} = +\infty.
$$
The key issue in the study of blow-up phenomena is to understand how and where explosion can take place.
 The blow-up is said to be of type I if we have that
 \be
\limsup_{t\to T} (T-t)^{\frac 1{p-1}}\|v(\cdot ,t)\|_{L^\infty(\Omega)} <+\infty
\nonumber \ee
and of type II if
 \be
\limsup_{t\to T} (T-t)^{\frac 1{p-1}}\|v(\cdot ,t)\|_{L^\infty(\Omega)} =+\infty.
\nonumber \ee
Type I means that the blow-up takes place like that of the ODE $v_t = v^p$, so that in the explosion mechanism
the nonlinearity plays the dominant role. In a related interpretation, the blow-up
``respects'' the natural scalings of the problem.
The second alternative is rare and far less understood. The delicate interplay of diffusion, nonlinearity
and geometry of the domain is responsible for that scenario.

\medskip 	
The role of the Sobolev critical exponent
$$p_S(d)=\begin{cases} \infty, & \text{if $1\le d\le 2$},\\ \frac{d+2}{d-2}, & \text{if $d\ge 3$}. \end{cases} $$
is well-known to be central in the possible types of blow-up for \equ{P1}.

\medskip
 When  $1<p< p_S(d)$  solutions can only have type I blow-up, as it was first  established by Giga and Kohn \cite{gk1} for the case of $\Omega$  convex, and in \cite{poon} in for a general domain. This is also the case for $p=p_S(d)$ and radial solutions of \equ{P1} \cite{fhv1}, or if $\Omega$ is  star-shaped \cite{cheng}. Type II blow-up radial sign-changing solutions exists for $p=p_S(4)$  \cite{fhv1,schweyer}.

 \medskip
Refined asymptotics of Type I blow-up together with constructions and classification results have been obtained in many works, we refer the reader to
\cite{collot2,gk2,gk3,matanomerle,mz,qs,velazquez} and references therein.

\medskip
Type I is expected to be in any reasonable sense the ``generic'' way in which blow-up takes place for any  $p>1$, see
\cite{collot1,fhv1,matanomerle,mm1}.

 \medskip
 Type II blow-up solutions
are much harder to be detected. The only examples known are for $d\ge 11$ and $p> p_{JL}(d)$
%Herrero and Vel\'azquez \cite{hv,hv1} found the first  example of Type II blow-up when $\Omega = \R^n$.
where $p_{JL}(d)$ is the  {\em Joseph-Lundgren exponent} \cite{jl} defined as
$$p_{JL}(d)=\begin{cases} \infty, & \text{if $3\le d\le 10$},\\ 1+{4\over d-4-2\,\sqrt{d-1}}, & \text{if $d\ge11$}. \end{cases} $$
%For $p> p_{JL}(n)$ the positive radial  solutions of
%$$ \Delta u + u^p =0\quad\hbox{in } \R^n$$
%are all ordered, and therefore stable as stationary states of \equ{P1} \cite{gnw}.
Herrero and Vel\'azquez \cite{hv,hv1} found a radial solution that blows-up with type II rate. The local profile locally resembles a time-dependent, asymptotically singular scaling of a positive radial solution of
\be\label{sss} \Delta w + w^p =0\quad\hbox{in } \R^d.\ee
See also \cite{mizoguchi} for the case of a ball, and \cite{collot4} for an arbitrary domain with the same profile profile when $p$
is in addition an odd integer.  A main ingredient in the constructions is the stability of radial solutions of \equ{sss}  whenever $p>p_{JL}(d)$ \cite{gnw}. No positive solution (radial or not) is stable for $p\le p_{JL}(d)$  \cite{farina}.

\medskip
In \cite{mm1}
Matano and Merle proved that in the radially symmetric case no Type II blow-up can take place if  $p_S< p\le  p_{JL}(d)$, a result that  precisely complements that for the Herrero-Velazquez range. Recently in \cite{collot1} an entire finite energy, axially symmetric type II blow-up solution  was built for $d\ge 12$ and  $p> p_{JL}(d-1)> p_{JL}(d)$.

\medskip
A question that has remained conspicuously open for many years is whether or not type II blow-up solutions of \equ{P1} can exist in
the Matano-Merle range $p_S(d)< p< p_{JL}(d)$. Such solutions must of course be non-radial. In this paper we prove that the answer is {\bf yes} in dimension $d\ge 7$ and $p= \frac{d+1}{d-3}=p_S(d-1)  $ in a class of domains with  axial symmetry.

\medskip
Let us identify $\R^d \simeq \C \times \R^{d-2}$ and consider the orthogonal transformations
\be
Q_\alpha (z) = ( e^{i\alpha}(z_1+iz_2) , z), \quad \pi_i(z,x') =  (z_1+iz_2, z_3,\ldots,-z_i,\ldots z_{d}).
\label{trans}\ee
We assume that $\Omega$ is a smooth, bounded domain with $0\not\in \Omega$,  that is invariant under this transformations:
$$
Q_\alpha (\Omega)= \Omega , \quad \pi_i(\Omega)=\Omega \foral \alpha\in \R, \quad  i=3,\ldots, d.
$$
In other words $\Omega$ is a radial domain in the first two coordinates, even in the remaining ones which does not contain the origin.
In $\Omega$ we consider the problem
\begin{align}\label{P2}
v_t   & = \Delta v+ v^{\frac {d+1}{d-3}}   \inn\ \Omega\times (0, T),  \\
%u   &\color{blue>0 \inn \Omega\times (0, \infty)\nonumber \\
v &  =0 \onn  \partial \Omega\times (0, T ) ,  \nonumber\\
v&  >0\inn\  \Omega\times (0, T) .\nonumber
\end{align}
Let
$
m= \inf \{ |z_1+iz_2|\ /\ (z_1+iz_2,0) \in \Omega\}  >0
$
so that the curve  $$\Gamma:= \{  (z_1+iz_2,0)\   / \ |z_1+iz_2| =m  \}   $$
 is a circle contained in $\pp\Omega$.

\begin{theorem}\label{teo1}
Let $d\ge 7$ and $\Omega$ a domain as described above. For any sufficiently small $T>0$,
there exists a smooth solution $v(z,t)$ of problem $\equ{P2}$ %which invariant under the transformations $\equ{trans}$
%and  has type II blow-up as $t\to T$. The solution
that remains uniformly bounded outside any neighborhood of the curve $\Gamma$ while
$$
\lim_{t\to T} (T-t)^{ \gamma  } \|v(\cdot,t)\|_{L^\infty(\Omega)}  >0 , \quad
\gamma = {(d-3) (d-4) \over 2(d-5)}.
$$
\end{theorem}
We notice that for $p= \frac{d+1}{d-3} $  we have
$
\frac 1{p-1}  = \frac {d-3} 4 < \gamma
$
so that $v$ exhibits type II blow-up.

\medskip
The construction provides very accurate information on the solution. The principle is very simple
We let $n= d-1$ and consider the standard Aubin-Talenti function \cite{t}
\be\label{bubble}
U(y) = \alpha_n \left (  \frac 1 {1+ |y|^2} \right )^{\frac{n-2}2}, \quad \alpha_n = (n(n-2))^{\frac 1{n-2}},
\ee
which is a positive solution of
$
\Delta U + U^{\frac{n+2}{n-2}} =0 .
$
The solution has the form $$v(z,t) \sim \frac 1{\la(t)^{\frac{n-2}2} } U\left (\frac x{\la(t)}\right) $$ where $x$ is the vector joining $z$ and its closest point to a circle of the form $(1+d(t))\Gamma$ contained in the domain, with $\la(t)\to 0$ and $d(t)\to 0$ as $t\to T$.
We have that the energy density $|\nn u(z,t)|^2$ concentrates in the form of a Dirac mass for the curve $\Gamma$, a phenomenon usually called bubbling. Bubbling at points
triggered by criticality is a feature known in several different contexts, including dispersive equations and geometric flows. There is a broad literature on that matter. We refer the reader for instance to \cite{av,CDM,CMR,dd,DMW,dkm,galk2,j1,kenigmerle,kst,MRR,RR} and their references.

\medskip
The phenomenon of higher dimensional boundary bubbling here discovered is definitely triggered by geometry and is entirely new in the diffusion setting.
It is worth mentioning that similar blow-up triggered by geometry of the boundary under axial symmetry has been numerically conjectured to hold
for the three dimensional Euler equation in \cite{hou}.

\medskip
In an elliptic context,  a result with resemblance to the current one was found in \cite{dmp}, methodologically connected with \cite{dkw}.
In fact we conjecture that a construction like the one here should be possible along a negatively curved closed geodesic of the boundary.
We believe that  geometry is essential and that in a convex domain the Matano-Merle range of exponents may still lead to non-existence of type II
blow-up.

\medskip
The 2-variable radial symmetry of the domain leads us to look for a solution with the same symmetry of a problem in a domain with one dimension less
at the critical exponent where point bubbling is obtained. This problem is methodologically challenging. For instance
 the method in the construction in \cite{schweyer}
of point  type II blow up in the radial sign-changing context in dimension 4, based on the pioneering work by Merle, Raphael, Rodnianski
\cite{MRR}, later applied to various blow-up problems, does not seem not apply here.  See \cite{collot4} for a difficult adaptation of that technique for a related problem in a non-radial setting, yet only valid for odd integer powers, which is never our case.

\medskip
We close this introduction by mentioning that our proof applies equally well to an exterior domain of the same nature. Besides, within the
symmetry class the phenomenon we obtain is codimension 1-stable (presumably highly unstable outside symmetry). We shall not elaborate in
that issue, which is a rather direct consequence of our construction.

\medskip
We devote the rest of this paper to the proof of Theorem \ref{teo1}.

\medskip
{\bf Notation.} \ \
We use the symbol $'' \, \lesssim \, ''$ to indicate $'' \, \leq C \, ''$, for a positive constant $C$, whose value may change from line to line, and also inside the same line, and which is independent of $t$ and $T$.

\medskip

\section{Scheme of the proof}\label{schema}

Let $d=n+1$, and consider the change of variables
$$
v(z_1 , \ldots , z_d , t) = u (\sqrt{z_1^2 + z_2^2} , z_3 , \ldots , z_d , t)
$$
for some $u = u(x_1 , \ldots , x_n, t)$.
In terms of $u$, solving \eqref{P2} in the class of functions $v$ that are invariant under the orthogonal transformations \eqref{trans} translates into solving
\begin{equation}
\label{pp} \begin{split}
u_t&= \Delta u+ {1\over x_1} {\partial u \over \partial x_1} +u^p \ \mbox{in} \ \quad \DD \times (0, T), \quad p={n+2 \over n-2}\\
 u&>0 \quad \mbox{in} \quad \DD \times (0,T), \quad u=0 \quad \mbox{on} \quad \partial \DD \times (0,T),
 \end{split}
\end{equation}
with $u(x_1, \ldots , x_j , \ldots , x_n ) = u (x_1, \ldots , -x_j , \ldots , x_n ) $ for any $j=2, \ldots , n$.
Here
 $\DD$ is the smooth bounded domain in $\R^n$
 defined as
$$
\Omega = \{ (z_1 , \ldots , z_d ) \in \R^d \, : \, (\sqrt{z_1^2 + z_2^2 } , z_3 , \ldots , z_d ) \in \DD \},
$$
 with the properties
 $$
 (x_1 , \ldots , x_j , \ldots , x_n ) \in \DD \quad \Longleftrightarrow \quad (x_1 , \ldots , -x_j , \ldots , x_n ) \in \DD , \quad j=2, \ldots , n,
$$
and
$$
\DD \cup \{ (x_1 , \bar x ) \in \R \times \R^{n-1} \, : \, \bar x = \bar 0 \} \not= \emptyset.
$$
With no loss of generality, we assume that
$$
\inf \{ r>0 \, : \, (r, \bar 0 ) \in \DD \} = 1.
$$

 \medskip
The result contained in Theorem \ref{teo1} is expressed in terms of Problem \eqref{pp} in the following

 \begin{theorem}\label{teo}
 For any $T$ small enough, there exists a finite time blow-up solution to Problem \eqref{pp} of the form
 \begin{equation*}\begin{split}
u(x,t) &= \la^{-{n-2 \over 2}} (t) \Biggl[ U\left({x-\xi (t) \over \la(t) } \right) -   U\left({x-\hat \xi (t) \over \la(t) } \right)\Biggl] \left( 1+ o(1)\right) + O(1)
\end{split}
\end{equation*}
where $O(1)$ is uniformly bounded and $o(1)\to 0$ uniformly as $t\to T$.
Here $U$ is defined in \eqref{bubble},
and
\begin{equation}\label{points}
\xi (t )  = (1+ d (t) ) \, {\bf e}_1, \quad \hat \xi (t) = (1-d (t)) \, {\bf e}_1, \quad {\mbox {where}} \quad {\bf e}_1 = (1, \bar 0)
\end{equation}
with
$$ d (t) =  (T-t) \, \left(1+ o(1) \right), \ \lambda =  (T-t)^{n-3 \over n-4} \, \left(1+ o(1) \right), \quad {\mbox {as}} \quad t\to T
$$
%and
%$$
%\sup_{x \in \DD} |g(x,t) | \to 0, \quad {\mbox {as}} \quad t\to T^- .
%$$
\end{theorem}

\medskip
We find the solution to Problem \eqref{pp} as predicted by Theorem \ref{teo} by constructing a sufficiently accurate approximation, and then an actual solution to the Problem as a {\it small} perturbation which is subtle to use in particular by the structure instability of the problem.
 Our solution has  the form
\begin{equation}\label{sol1}
u(x,t) = W_2  (x,t) + {\bf w} (x,t)
\end{equation}
where $W_2$ is an explicit approximation whose expression
 encodes the predicted asymptotic behavior as $t\to T$. Here ${\bf w}$ is a small correction in some appropriate topology.

\medskip
In the rest of this section we describe $W_2(x,t)$ and the method of construction of an actual solution near $W_2$ which we call {\it the inner-outer gluing method}.

\medskip
{\it Construction of the approximation $W_2(x,t)$ }. \ \ We
introduce two scalar functions $d$, $\lambda : (0,T) \to \R$, expressed respectively as
\begin{equation}\label{def1}
d(t) = d_0 (t) + d_1 (t) , \quad \lambda (t) = \lambda_0 (t) + \lambda_1 (t) ,
\end{equation}
where $d_0$ and $\la_0$ are explicitly given by
\begin{equation}\label{def2}
d_0 (t) = (T-t) , \quad \lambda_0 (t) = \ell (T-t)^{1+{1\over n-4}}
\end{equation}
with $\ell$ a positive constant that we will define later.
The functions $\la_1 $ and $d_1$ are thought as parameter functions to be determined. For the moment, we assume that $\la_1$ and $d_1$ are controlled by  $\la_0$ and $d_0$ in the whole interval $(0,T)$, in the following sense. For  any scalar function $h(t)$, $t \in (0, T)$, and any real number $\delta $,  $\| h \|_\delta$  stands for the weighted $L^\infty$-norm defined as
\begin{equation}\label{norma}
\| h \|_\delta = \| (T-t)^{-\delta} h(t) \|_{L^\infty (0,T) }.
\end{equation}
We assume that
\begin{equation}\label{n1}\begin{split}
\la_1 (t) := \int_t^T & \dot \la_1 (s) \, ds, \quad d_1 (t) := \int_t^T \dot d_1 (s) \, ds, \quad {\mbox {with}} \quad\\
 & \| \dot d_1 \|_{1+\sigma \over n-4} + \| \dot \la_1 \|_{1+\sigma \over n-4} \lesssim 1,
\end{split}
\end{equation}
where $\sigma $ is fixed in the range
\begin{equation}\label{sigmadef}
\sigma \in ({1\over 2},1).
\end{equation}
In fact, we can think at $\sigma$ as close to $1$.
The final time $T>0$ will be chosen to be small enough so that $d(t) > 0$, and $\la (t) >0$, in the whole interval $t \in (0, T)$.

\medskip
As in the  statement of Theorem \ref{teo}, we proceed with the first step in the construction of $W_2$ in \eqref{sol1} and we  introduce
\begin{equation}\label{appro1}
W_1 [\la_1 , d_1] (x,t) = W_0 (x,t) - \bar W_0 (x,t)
\end{equation}
where
$$
W_0 (x,t) =\la^{-{n-2 \over 2}}  (t)  U\left( {x-\xi (t) \over \la (t)} \right) , \quad \bar W_0 (x,t) = \la^{-{n-2 \over 2}}  (t) U \left( {x-\hat \xi (t) \over \la (t) } \right),
$$
with $U$ given by \eqref{bubble}, and the points $\xi$, $\hat \xi$  described in \eqref{points}. We recall that $U$ solves
\begin{equation}\label{bubbleeq}
\Delta u + u^{n+2 \over n-2} = 0 , \quad {\mbox {in}} \quad \R^n.
\end{equation}

\medskip
Since $d(t) >0$ for all $t \in (0, T)$, we see that $\xi (t) \in \DD$, $\hat \xi (t) \not\in \DD$, for any $t \in (0,T)$. In fact,  since $(1, \bar 0 ) \in \partial \DD$, the point $\hat \xi$ is the reflection of $\xi$ through the boundary. In other words, ${\xi (t) + \hat \xi(t) \over 2} = (1, \bar 0)$.   The radial symmetry of $U$ implies that $W[\la_1 , d_1] (1,\bar 0) =0$, for any possible election of $\la_1 $ and $d_1$.

\medskip
The way to establish whether $W_1$ is a good approximation is to measure the size of the error $S[W_1] (x,t)$, where
$$
S[u] (x,t) = -u_t + \Delta u + {1\over x_1} {\partial u \over \partial x_1} + u^p.
$$
Formally, one sees that, locally around  a small neighborhood of $\xi$, the error $S[W_1]$ looks like, in the expanded variable $ y ={ x-\xi \over \la}$,
\begin{equation}
\begin{split}\label{vvn}
S[W_1] (\xi + \la y ,t) &\sim \la^{-{n\over 2}} \left[ \dot d +{1\over 1+ d +\la y_1 } \right] Z_1 \left(y  \right) \quad  \quad \quad \\
&+ \la^{-{n\over 2}} \dot \la Z_0 \left( y \right) - {p \alpha_n \over 2^{n-2}}  \la^{-{n+2 \over 2}} \, \left({\la \over d} \right)^{n-2} \,  U^{p-1} (y) .
\end{split}
\end{equation}
We refer to \eqref{bubble} for the definition of the constant $\alpha_n$ and to  \eqref{ee1} for the precise expression of $S[W_1]$.

The functions $Z_1$ and $Z_0$ that appear in \eqref{vvn} are
\begin{equation}\label{elk}
Z_0 (y ) = {n-2 \over 2} U(y) + \nabla U(y) \cdot y , \quad Z_1 (y) = {\partial U \over \partial y_1} (y),
\end{equation}
and they are the only bounded solutions to the linearized equation of \eqref{bubbleeq} around $U$
\begin{equation}\label{L0}
L_0 (\phi ) := \Delta \phi + p U^{p-1} \phi = 0 , \quad {\mbox {in}} \quad \R^n
\end{equation}
in the class of functions that are even in the variable $y_j$, for $j=2, \ldots , n$.

\medskip
The definition of $d_0 = d_0 (t)$ in \eqref{def1}-\eqref{def2} makes the biggest part of the function inside brackets in the first term in \eqref{vvn} at the point $y=0$ equals to zero, since
$$
\dot d_0 (t) + 1 = 0, \quad t \in (0,T), \quad d_0 (T )=0.
$$

\medskip
With this choice of $d_0$, the definition of $\la_0 $ in \eqref{def1}-\eqref{def2} makes the integration of the second and third terms in \eqref{vvn} against $Z_0$ in $\R^n$ equals to zero. Indeed, $\la_0$ is the solution to the ordinary differential equation
\begin{equation}\label{defla0}
\la_0 \dot \la_0 \left( \int_{\R^n} Z_0^2 (y) \, dy \right) - {p\, \alpha_n \over 2^{n-2} } \left( {\la_0 \over d_0} \right)^{n-2} \left( \int_{\R^n} U^{p-1} (y) \, Z_0 (y) \, dy \right) = 0,
\end{equation}
with $\la_0 (T) = 0$,
provided the number $\ell $ in \eqref{def2} is given by
\begin{equation}\label{def3}
\ell = \left[ {n-3 \over n-4} {2^{n-1} \over \alpha_n (n-2)} {\int_{\R^n} Z_0^2 (y) \, dy \over
\int_{\R^n} U^p (y) \, dy } \right]^{1\over n-4},
\end{equation}
since $-p \int_{\R^n} U^{p-1} (y) Z_0 (y) \, dy = {n-2 \over 2} \int_{\R^n} U^p (y) \, dy.$

\medskip
A rigorous description of the error $S[W_1]$ in a region close to $\xi$  is contained in Lemma \ref{lemma1} in  Section \ref{tre}. The main part of $S[W_1]$ turns out to be an explicit function of $(x,t)$, independent of $d_1$ and $\la_1$. It is thus easy to correct $W_1$ to cancel the biggest part of the error, so we end up with a final approximation we called $W_2 = W_2 [\la_1 , d_1]$. The description of the second error $S[W_2]$ in a region close to the point $\xi$ is contained in Lemma \ref{lemma2},  while the description of part of the error $S[W_2]$ far from $\xi$ is estimated in Lemma \ref{lemma3}. In Lemma \ref{lemma3}, we also provide a description of $W_2$ on the boundary $\partial \Omega$, which unfortunately is not identically zero. The correction of the boundary term and the construction of an actual solution to the equation is done in the second step of our argument, through the {\it inner-outer gluing method}.

\medskip
{\it Inner-outer gluing method}. \ \ This method is a procedure to find the function ${\bf w}$ in \eqref{sol1}.
We expect that the function ${\bf w}$ corrects the approximation $W_2$ in a region far from the point $\xi$, adjusting of course the boundary conditions, and at the same time in a region close to $\xi$.

\medskip
To organize this double role for ${\bf w}$, we introduce  a smooth cut-off function $\eta $ with $\eta(s) =1$ for $s<1$ and $=0$ for $s>2$,
and we define
\begin{equation}\label{cutoff}
\eta_R (x,t) = \eta \left( {|x-\xi| \over R \la} \right).
\end{equation}
We will take the radius $R$ a large number, independent of $t$, but dependent on $T$.
We write
\begin{equation}\label{defw}
{\bf w} (x,t) = \psi (x,t) + \eta_R (x,t)  \Phi (x,t).
\end{equation}
In this decomposition, the term $\psi$ is mainly influenced from the region far from $\xi$, while $\Phi$ reflects what is going on close to $\xi$.

\medskip
In order that $u= u(x,t)$ defined in \eqref{sol1} to be an actual solution to problem \eqref{pp},   the function ${\bf w}$ has to satisfy
\begin{equation}\label{s1} \begin{split}
{\bf w}_t&= \Delta {\bf w} + p W_2^{p-1} {\bf w} + {1\over x_1} {\partial {\bf w} \over \partial x_1} + S[W_2] (x,t) + N ({\bf w}) , \quad (x,t) \in \DD \times (0,T)\\
{\bf w}&= - W_2 , \quad (x,t) \in \partial \DD \times (0,T).
\end{split}
\end{equation}
where
\begin{equation}\label{s2}
N({\bf w}) = (W_2 + {\bf w} )^p - W_2^p - p W_2^{p-1} {\bf w}.
\end{equation}
Thanks to \eqref{defw}, we proceed to decompose problem \eqref{s1} into an {\it outer} and a {\it inner} problem.

\medskip
Let $R' >R$ so that $\eta_{R'} \eta_R = \eta_R$. The equation in \eqref{s1} is written explicitly in terms of $\psi$ and $\Phi$ as follows
\begin{equation*}\begin{split}
\psi_t &+ \underline{\eta_R \Phi_t} + (\eta_{R})_t \Phi = \Delta \psi + {1\over x_1} {\partial \psi \over \partial x_1}  +\underline{\eta_R  ( \Delta \Phi + {1\over x_1} {\partial \Phi \over \partial x_1}  )} \\
&+ 2 \nabla \eta_R \nabla \Phi + \Delta \eta_R \Phi + {1\over x_1} {\partial \eta_R \over \partial x_1}  \Phi\\
&+ \underline{ p \left[ \la^{-{n-2 \over 2}} U ({x-\xi \over \la} ) \right]^{p-1} \eta_{R'} \eta_R (\psi + \Phi)} \\
&+ p \left[ \la^{-{n-2 \over 2}} U ({x-\xi \over \la} ) \right]^{p-1} \eta_{R'}  (1-\eta_R) \psi \\
&+ p \left[ W_2^{p-1} - [ \la^{-{n-2 \over 2}} U ({x-\xi \over \la} ) ]^{p-1} \right] \eta_{R'} (\psi + \eta_R \Phi)\\
&+ p W_2^{p-1} (1-\eta_{R'} ) (\psi + \eta_R \Phi) + N [{\bf w} ] + \underbrace{\bar E_2+\underline{E_2 \eta_R }}_{:= S[W_2]}.
\end{split}
\end{equation*}
Here we have decomposed the error $S[W_2]$ into its principal part $E_2$ multiplied by the cut off $\eta_R$,
\begin{equation}\label{gia}
S[W_2] = E_2 \eta_R + \bar E_2
\end{equation}
leaving all the rest into a term named $\bar E_2$.
Observe that the terms which are underlined all go with the cut off function $\eta_R $ in front.
Define
\begin{equation}\label{defV}\begin{split}
V(x,t)&=  p \left[ \la^{-{n-2 \over 2}} U ({x-\xi \over \la} ) \right]^{p-1} \eta_{R'}  (1-\eta_R)\\
&+p \left[ W_2^{p-1} - [ \la^{-{n-2 \over 2}} U ({x-\xi \over \la} ) ]^{p-1} \right] \eta_{R'} \\
& +p W_2^{p-1} (1-\eta_{R'} ) .
\end{split}
\end{equation}
Then we observe that ${\bf w}$ defined in \eqref{defw} solves \eqref{s1} if the pair $(\psi , \Phi)$ solve the following system of coupled equations
\begin{equation}\label{sout}\begin{split}
\psi_t &= \Delta \psi  + {1\over x_1} {\partial \psi \over \partial x_1}  + V \psi + \left( \Delta -{\partial \over \partial t} \right) \eta_R \, \Phi + 2\nabla \Phi \nabla \eta_R  + {1\over x_1} {\partial \eta_R \over \partial x_1}  \Phi\\
&+p \left[ W_2^{p-1} - [ \la^{-{n-2 \over 2}} U ({x-\xi \over \la } ) ]^{p-1} \right] \eta_{R'} \eta_R \Phi\\
&+ N [{\bf w} ] + \bar E_2 \quad {\mbox {in}} \quad \DD \times (0,T)\\
\psi &= -W_2 , \quad {\mbox {on}} \quad \partial \DD \times (0,T),
\end{split}
\end{equation}
and
\begin{equation}\label{sin} \begin{split}
\Phi_t &= \Delta \Phi + p \left[ \la^{-{n-2 \over 2}} U ({x-\xi \over \la } ) \right]^{p-1}   ( \Phi + \psi ) + {1\over x_1} {\partial \Phi \over \partial x_1} \\
&+E_2  \quad {\mbox {in}} \quad B(\xi , 2R\la ) \times (0,T).
\end{split}
\end{equation}
Problem \eqref{sout} is referred to as the {\it outer problem}: $\psi$ adjusts the boundary conditions, and takes care of the part of the error far from the concentration point $\xi$.

Problem \eqref{sin} is referred to as the {\it inner problem}: $\Phi$ adjusts the error close to $\xi$.

\medskip
\medskip
To solve the outer and inner problems \eqref{sout} and \eqref{sin}, we  proceed as follows. For given parameters $\la, d  $ and  functions $\Phi$ fixed  in a suitable range,  we solve for $\psi$ Problem \equ{sout}, for any small and smooth initial condition $\psi_0 (x)$, in the form of a (nonlocal) operator $\psi =  \Psi(\la,d,  \Phi)$, provided the radius $R$ in \eqref{cutoff} is large enough and the final time $T$ is small enough. We  solve it developing a linear theory for an operator which resembles the characteristics of the heat equation. This is done in full details in Section \ref{outer}.

We then replace the $\psi$ we found into the inner problem  \equ{sin}.
In order to get a cleaner expression for problem \eqref{sin},
 it is convenient to perform two changes of variable for the function $\Phi$.  First, we perform a change of variable in the {\it space} variable, by setting
\begin{equation}\label{defPhi}
 \Phi (x,t) = \la^{-{n-2 \over 2}} \phi \left( {x-\xi \over \la } , t\right).
\end{equation}
In terms of $\phi$, equation \eqref{sin} gets the form
\begin{equation}\label{sin1} \begin{split}
\la^2 \phi_t &= L_0 (\phi)  + p \la^{n-2 \over 2} U^{p-1} \psi (\la y + \xi , t) +\la^{n+2 \over 2} E_2 (\la_0 y + \xi , t)  \\ &+ B[ \phi] \quad {\mbox {in}} \quad B(0 , 2R ) \times (0,T),
\end{split}
\end{equation}
where $L_0$ is the linearized equation associated to the bubble $U$, introduced in \eqref{L0}, that we recall
$L_0 (\phi) =\Delta \phi + p U^{p-1}   \phi$, and
\begin{equation}\label{defB}
B[\phi ] = \la \dot \la \left[ {n-2 \over 2} \phi (y, t) + \nabla \phi(y, t) \cdot y \right] + \left[ \la \dot d + {\la \over \la y_1 + \xi } \right] {\partial \phi \over \partial
 y_1} (y, t)
\end{equation}
A second change of variable, in the {\it time} variable, is to define
\begin{equation}\label{deftau}
{dt \over d\tau } = \la^2 (t), \quad \tau (t) = {n-4 \over (n-2) \ell } (T-t)^{-1-{2\over n-4}} \left(1+ o(1)\right) , \quad {\mbox {as}} \quad t \to T
\end{equation}
where $\ell$ is the constant defined in \eqref{def3}. With this change in the time variable, equation \eqref{sin1} becomes
\begin{equation}\label{sin2} \begin{split}
 \phi_{\tau} &= \Delta \phi + p U^{p-1}   \phi   + H[\la , d , \phi, \psi] (y,\tau ) \quad {\mbox {in}} \quad B(0, 2R ) \times (\tau_0 , \infty),
\end{split}
\end{equation}
for $\tau_0 =\tau (0) $ and
\begin{equation}\label{defH}
\begin{split}
H[\la , d , \phi, \psi] (y,\tau ) &= p \la^{n-2 \over 2} U^{p-1} \psi (\la y + \xi , t (\tau )) +\la^{n+2 \over 2} E_2 (\la y + \xi , t (\tau )) \\
  &+ B[ \phi]
\end{split}
\end{equation}
Let us discuss how we treat Problem \eqref{sin2}. The linear operator $L_1 (\phi ):= -\phi_\tau + L_0 (\phi)$ is certainly not invertible, being all $\tau$-independent elements of the kernel of $L_0$ also elements of the kernel of $L_1$. Thus, for solvability, one expects some orthogonality conditions to hold. Not only this. The solution $\phi$ we look for cannot grow exponentially in time. Recall that
$L_0$ has a
 positive radially symmetric bounded eigenfunction $Z$
associated to the only negative eigenvalue $\mu_0$ to the problem
\begin{equation}
\label{eigen0}
L_0( \phi ) + \mu \phi = 0 , \quad  \phi \in L^\infty(\R^n).
\end{equation}
It is known that $\mu_0 $ is a simple eigenvalue and   that $Z$
 decays like
$$Z (y) \sim  |y|^{-\frac{n-1}2} e^{-\sqrt{|\mu_0 |}\,  |y|}  \ass |y| \to \infty. $$
To avoid exponential grow in time due to this instability,
we  construct a solution to \eqref{sin2} in the class of functions that are parallel to $Z$ in the initial time $\tau_0$.

\medskip
To be more precise, we can  construct a solution to the initial value problem
\begin{equation}\label{sin3} \begin{split}
 \phi_{\tau} &= \Delta \phi + p U^{p-1}   \phi +  H [\la , d, \phi, \psi] (y,\tau )
   \quad {\mbox {in}} \quad B(0, 2R ) \times (\tau_0 , \infty), \quad \\
   \phi(y,\tau_0) &= e_0Z(y)  \inn B (0,2R),
\end{split}
\end{equation}
for some constant $e_0$.
While no boundary conditions are specified, we shall request suitable time-space decay rates and, as already mentioned, some orthogonality conditions  on the right-hand side $H[\la, d , \phi, \psi]$.
In other words,  one has solvability for \eqref{sin3} provided that
the following orthogonality conditions
\begin{equation}\label{uffa1}
\int H[\la  , d , \phi , \psi]  Z_i (y) \, dy = 0 , \quad i=0,1 \quad \forall t
\end{equation}
are fulfilled.  It is at this point that
we choose the parameters $\la$ and $d$ (as functions of the given $\phi$) in such a way that these orthogonality   conditions are satisfied.  This is done in Section \ref{parfun}, for any $R$ (see \eqref{cutoff}) large enough, and any final time $T$ small enough.

In Sections \ref{inner} we solve the inner problem \eqref{sin3}: it is at this point that we find that there exists $R$ sufficiently large for that, for any final time $T$ small enough (or equivalently $\tau_0$ large enough), the inner problem is solvable.
 We remark that the (small) initial condition required for  $\phi$ should lie on a certain manifold locally described as a translation of the hyperplane orthogonal to $Z(y)$. This constraint defines a {\em codimension $1$ manifold}  of initial conditions which describes those for which the expected asymptotic bubbling behavior is possible.

\medskip
In summary, the inner-outer gluing procedure allows us to show that: for any small and smooth initial condition $\psi_0$ for Problem \eqref{sout}, we find a solution $\psi$  to \eqref{sout}, $\la $, $d$ solutions to \eqref{uffa1}, and $\phi$ solution to \eqref{sin3}, with initial condition belonging to a $1$-codimensional space, so that
$W_2 (x,t) + {\bf w} (x,t) $ defined in \eqref{sol1}-\eqref{defw} is a solution to \eqref{pp} with the expected asymptotic bubbling behavior.

\medskip
The rest of the paper is devoted to prove rigorously what we have described so far.

\section{Construction of a first approximation} \label{tre}

We start with the description of the error function associated to the first approximation $W_1$, introduced in \eqref{appro1}.
We recall the definition of the error function
$$
S[u] (x,t) = -u_t + \Delta u + {1\over x_1} {\partial u \over \partial x_1} + u^p.
$$
A direct computation gives
\begin{equation}\label{ee1}
\begin{split}
S[W_1] (x,t) &=\underbrace{\la^{-{n\over 2}} \left[ \dot d +{1\over x_1} \right] Z_1 \left({x-\xi \over \la} \right)}_{:=e_1 (x,t)} \\
&+ \underbrace{\la^{-{n\over 2}} \dot \la Z_0 \left( {x-\xi \over \la } \right) - p W_0^{p-1} \bar W_0 }_{:= e_2 (x,t)}\\
&+\underbrace{\la^{-{n\over 2}} \left[ \dot d - {1\over x_1} \right] Z_1 \left( {x-\hat \xi \over \la} \right) -\la^{-{n\over 2}} \dot \la Z_0 \left( {x-\hat \xi \over \la} \right)}_{:=e_3 (x,t)} \\
&+\underbrace{ \bar W_0^p + (W_0 - \bar W_0 )^p - W_0^p + p W_0^{p-1} \bar W_0}_{:=e_4 (x,t)}.
\end{split}
\end{equation}

\medskip
We shall see that the main parts of the error function $S[W_1](x,t)$ are contained in the terms $e_1 $ and $e_2$. Observe also that
 the term $e_4$ depends only on $\la_1 $ and $d_1$, but it does not depend on $\dot \la_1$, nor on $\dot d_1$, while the term $e_3$ depends on all parameter functions $\la_1 $, $d_1$, $\dot \la_1$ and $\dot d_1$.

Next Lemma contains a description of the error function $S[W_1] (x,t )$ in a region close to $\xi$.

\medskip

\begin{lemma}\label{lemma1}
Assume the functions $\la_1$ and $d_1$ satisfy \eqref{n1}, and that $T$ is small.
Let $\delta >0$ be a small fixed number and $y = {x-\xi \over \la}$. In the region $|x-\xi |< \delta d$, the error of approximation $S[W_1] (x,t)$ can be described as
follows
\begin{equation}\label{app1}\begin{split}
\la^{n+2 \over 2} S[W_1] (x,t)& = E_0 (y,t) + E_\la [\la_1 , \dot \la_1 , d_1] (y,t) \\
&+ E_d [d_1 , \dot d_1 , \la_1 ] (y,t) + E[\la_1 , \dot \la_1 , d_1 , \dot d_1] (y,t)
\end{split}
\end{equation}
where
\begin{equation*}\begin{split}
E_0 (y,t) &=\la_0 \dot \la_0 Z_0 (y) - {p\, \alpha_n \over 2^{n-2} } \left({\la_0 \over d_0} \right)^{n-2} U^{p-1} (y),\\
E_\la  [\la_1 , \dot \la_1 , d_1] (y,t)& =(\la \dot \la_1 + \dot \la_0 \la_1 ) \, Z_0 (y) \\
&-{p \, (n-2) \, \alpha_n \over 2^{n-2}}  \left({\la_0 \over d_0} \right)^{n-2} \left[ {\la_1 \over \la_0} - {d_1 \over d_0}  \right] \left[ 1+ q_1 ({\la_1 \over \la_0} , {d_1 \over d_0} ) \right] U^{p-1} (y)\\
E_d [d_1 , \dot d_1 , \la_1] (y,t)& = \la \left[ \dot d_1 - {d_0 + d_1 + \la y_1 \over 1+ d + \la y_1} \right] Z_1 (y)\\
&+ {p \, (n-2) \, \alpha_n \over 2^{n-1}} \left({\la \over d} \right)^{n-1} U^{p-1} (y) \, y_1\\
E[\la_1 , \dot \la_1 , d_1 , \dot d_1] (y,t)&=\la \dot d_1 \left( {\la_0 \over d_0} \right)^{n-1} f(y, {\la_1 \over \la_0 }  , {d_1 \over d_0} ) - \la \dot \la_1 \left( {\la_0 \over d_0} \right)^{n-2} f(y, {\la_1 \over \la_0 }  , {d_1 \over d_0} )\\
&+ \left( {\la_0 \over d_0} \right)^{n+2} f(y, {\la_1 \over \la_0 }  , {d_1 \over d_0} ).
\end{split}
\end{equation*}
Here $f = f(y, {\la_1 \over \la_0 }  , {d_1 \over d_0} )$ denotes a generic function, which is smooth and bounded for $y$ in the considered region, and for $\la_1 $ and $d_1$ satisfying \eqref{n1}, whose expression changes from line to line. With $q_1$ we denote a generic smooth real function, with the property that $q_1 (0,0) = 0$, and $\nabla q_1 (0,0) \not= 0$.

\end{lemma}

\medskip
\begin{remark}
 A close look at the proof of Lemma \ref{lemma1} shows that the three functions $E_0$, $E_\la$ and $E_d$ originate from the terms $e_1$ and $e_2$ in \eqref{ee1}, which, as already mentioned, are the main terms of $S[W_1]$.
\end{remark}

\medskip

\begin{proof}
Let $\delta >0$ be a small fixed number. To analyze $S[W_1] (x,t)$ in the region $|x-\xi | < \delta d$, we introduce the variable
$y= {x-\xi \over \la}$ and we define
$$
E_1 (y,t ) = \la^{n+2 \over 2} S[W_1] (\xi + \la y ,t).
$$
With abuse of notation, we will write $e_j (y,t) = e_j (\xi + \la y ,t)$. The definition of $d_0$ in \eqref{def2} gives that $\dot d_0 + 1 = 0$ in $(0,T)$, which simplifies the first term $e_1$ as follows
\begin{equation}\label{le0}
\la^{n+2 \over 2} e_1 (y,t) = \la \left[ \dot d_1 - {d_0 + d_1 + \la y_1 \over 1+ d + \la y_1} \right] Z_1 (y).
\end{equation}
We refer to \eqref{elk} for the definition of $Z_1 (y)$.
Let us now describe $e_2$. In the region we are considering, $|y| < \delta {d \over \la}$, we observe that
$$
\la^{n-2 \over 2} \bar W_0 (\xi + \la y) = { \alpha_n \over 2^{n-2 }} \left( {\la \over d} \right)^{n-2} \left[ 1 - {n-2 \over 2} y_1 {\la \over d} +
O(1+ |y|^2 ) q_2 \left( {\la \over d} \right) \right]
$$
where $q_2$ denotes a smooth function with the properties that $q_2 (0) = q_2' (0) = 0$, $q_2'' (0) \not= 0$. With this in mind, we get
\begin{equation}  \label{le1}\begin{split}
\la^{n+2 \over 2} e_2 (y,t) &= \la \dot \la Z_0 (y) - {p\, \alpha_n \over 2^{n-2}} \left({\la \over d} \right)^{n-2} U^{p-1} (y)\\
&+  {p \, (n-2) \, \alpha_n \over 2^{n-1}} \left({\la \over d} \right)^{n-1} U^{p-1} (y) \, y_1 + R [\la , d] (y,t) \left({\la \over d} \right)^{n}
\end{split}
\end{equation}
where $R[\la, d] (y,t)$ depends smoothly on $\la$ and $d$, it does not depend on $\dot \la $, nor on $\dot d$, and satisfies the uniform estimates
\begin{equation}\label{inter}
\left| R[\la , d] (y,t) \right| \leq {C \over 1+|y|^2},
\end{equation}
for some constant $C$, independent of $t$ and $T$.
Replacing \eqref{defla0} in \eqref{le1}, we can write
\begin{equation}\label{le2}\begin{split}
\la^{n+2 \over 2} e_2 (y,t) &= \la_0 \dot \la_0 Z_0 (y) - {p\, \alpha_n \over 2^{n-2} } \left({\la_0 \over d_0} \right)^{n-2} U^{p-1} (y) \\
&+ (\la \dot \la_1 + \dot \la_0 \la_1 ) \, Z_0 (y) \\
&-{p \, (n-2) \, \alpha_n \over 2^{n-2}}  \left({\la_0 \over d_0} \right)^{n-2} \left[ {\la_1 \over \la_0} - {d_1 \over d_0} \right] \left[ 1+ q_1 ({\la_1 \over \la_0} ) + q_1 ( {d_1 \over d_0} ) \right] U^{p-1} (y)\\
&+  {p \, (n-2) \, \alpha_n \over 2^{n-1}} \left({\la \over d} \right)^{n-1} U^{p-1} (y) \, y_1  + \left({\la \over d} \right)^{n} R[\la, d] (y,t)  ,
\end{split}
\end{equation}
where $R$ depends smoothly on $\la$ and $d$, it does not depend on $\dot \la $, nor on $\dot d$, and satisfies the uniform estimate \eqref{inter}.
In order to describe $\la^{n+2 \over 2} e_3 (y,t)$, we observe that in the region we are considering, we have
$$
Z_1 (y +2 {d \over \la } {\bf e}_1 ) = \left( {\la_0 \over d_0} \right)^{n-1} f(y, {\la_1 \over \la_0 }  , {d_1 \over d_0} )
$$
and
$$
Z_0 (y +2 {d \over \la } {\bf e}_1 ) = \left( {\la_0 \over d_0} \right)^{n-2} f(y, {\la_1 \over \la_0 }  , {d_1 \over d_0} )
$$
where $f = f(y, {\la_1 \over \la_0 }  , {d_1 \over d_0} )$ denotes a generic function, which is smooth and bounded, whose expression changes from line to line. So, we get
\begin{equation}\label{le3}\begin{split}
\la^{n+2 \over 2} e_3 (y,t) &= \la \left[ \dot d_1 - 2 -{d+\la y_1 \over 1+ d + \la y_1} \right] \left( {\la_0 \over d_0} \right)^{n-1} f(y, {\la_1 \over \la_0 }  , {d_1 \over d_0} )\\
&-\la [\dot \la_0 + \dot \la_1 ]\left( {\la_0 \over d_0} \right)^{n-2} f(y, {\la_1 \over \la_0 }  , {d_1 \over d_0} )\\
&= \la \dot d_1 \left( {\la_0 \over d_0} \right)^{n-1} f(y, {\la_1 \over \la_0 }  , {d_1 \over d_0} ) - \la \dot \la_1 \left( {\la_0 \over d_0} \right)^{n-2} f(y, {\la_1 \over \la_0 }  , {d_1 \over d_0} )\\
&+ \left( {\la_0 \over d_0} \right)^{2(n-2)} f(y, {\la_1 \over \la_0 }  , {d_1 \over d_0} ).
\end{split}
\end{equation}
We finally observe that, for $|x-\xi | < \delta d$, we have
$$
\la^{n+2 \over 2} \bar W_0^p (\xi + \la y , t) = \left( {\la_0 \over d_0} \right)^{n+2} f(y, {\la_1 \over \la_0 }  , {d_1 \over d_0} )
$$
and, for $n=6$
\begin{equation*}\begin{split}
\la^{n+2 \over 2} &\left[ (W_0 - \bar W_0 )^p - W_0^p + p W_0^{p-1} \bar W_0 \right] (\xi + \la y ,t) = \left( {\la_0 \over d_0} \right)^{2(n-2)}  f(y, {\la_1 \over \la_0 }  , {d_1 \over d_0} )
\end{split}
\end{equation*}
while for $n\geq 7$
\begin{equation*}\begin{split}
\la^{n+2 \over 2} \left[ (W_0 - \bar W_0 )^p - W_0^p + p W_0^{p-1} \bar W_0 \right] (\xi + \la y ,t)&= \left( {\la_0 \over d_0} \right)^{n+2} f(y, {\la_1 \over \la_0 }  , {d_1 \over d_0} ).
\end{split}
\end{equation*}
We thus conclude that
\begin{equation}\label{le4}
\la^{n+2 \over 2} e_4 (y,t) = \left( {\la_0 \over d_0} \right)^{n+2} f(y, {\la_1 \over \la_0 }  , {d_1 \over d_0} )
\end{equation}
where $f = f(y, {\la_1 \over \la_0 }  , {d_1 \over d_0} )$ is  a smooth bounded function.
Putting together \eqref{le0}-\eqref{le1}-\eqref{le2}-\eqref{le3}-\eqref{le4}, and using the fact that $n\geq 6$, we obtain
\eqref{app1}.
\end{proof}

Observe that the function $E_0$ in \eqref{app1} is an explicit function of $x$ and $t$, and it does not depend on the parameter functions $\la_1$, and $d_1$. It is convenient to slightly modify the approximate solution $W_1$, adding a correction that will eliminate the term $E_0$ in the error. To this purpose, we write
$$
E_0 (y,t) = \left({\la_0 \over d_0 } \right)^{n-2} \pi (y), \quad \pi (y) = {p\, \alpha_n \over 2^{n-2}} \left[ {\int_{\R^n} U^{p-1} (y) Z_0 (y) \, dy \over \int_{\R^n } Z_0^2 (y) \, dy } \, Z_0 (y) - U^{p-1} (y) \right].
$$
Let $h= h(y)$ be the radially symmetric, fast decaying solution to
$$
\Delta h + p U^{p-1} h = \pi, \quad {\mbox {in}} \quad \R^n,
$$
defined by the variation of parameters formula as follows. We denote by  $\tilde Z$ a radial solution to $\Delta \tilde Z + p U^{p-1} \tilde Z=0$
which is linearly independent to $Z_0$. One has that $\tilde Z (r) \sim r^{2-n}$ as $r\to 0$, while $\tilde Z (r) \sim 1$ as $r \to \infty$. Then $h$ is given by
$$
h(r) = c Z_0 (r) \int_0^r \tilde Z (s) \pi (s) s^{n-1} \, ds - c \tilde Z (r) \int_r^\infty Z_0 (s) \pi (s) s^{n-1} \, ds, \quad r= |y|,
$$
for some constant $c$. One sees that
\begin{equation}\label{es1}
h(|y|) = O (|y|^{-2} ) , \quad {\mbox {as}} \quad |y| \to \infty.
\end{equation}
Define
$$
w(x,t) = \la^{-{n-2 \over 2}} h \left( {x-\xi \over \la } \right), \quad {\mbox {and}} \quad
\bar w(x,t) = \la^{-{n-2 \over 2}} h \left( {x-\hat \xi \over \la } \right).
$$
Observe that
\begin{equation}\label{es2}
\Delta w + p \left( \la^{-{n-2 \over 2}} U\left( {x-\xi \over \la } \right) \right)^{p-1} w = \la^{-{n+2 \over 2}} \pi \left( {x-\xi \over \la } \right).
\end{equation}
A new approximate solution is defined to be
\begin{equation}\begin{split} \label{appro2}
W_2 [\la_1 , d_1] (x,t ) &= W_1 [\la_1 , d_1] (x,t) \\
 &-  \left({\la_0 \over d_0} \right)^{n-2} \underbrace{ \left[ w (x,t) - \bar w(x,t) \right] \, \eta \left( { |x-\xi| \over b \, d_0 } \right) }_{:= W(x,t)}, \end{split}
\end{equation}
where $W_1$ is the function defined in \eqref{appro1}, and  $\eta $ is the smooth cut-off function we see in \eqref{cutoff}, with $\eta(s) =1$ for $s<1$ and $=0$ for $s>2$. The number $b>0$ is chosen to be small and fixed in such a way that
 $\eta \left( { |x-\xi| \over b \, d_0 } \right) \equiv 0$ for any $x \in \partial \DD$, for any $t \in [0, T)$. Such a choice is possible
 thanks to \eqref{n1}, for any $T$ small.

 \medskip
 A direct computation gives that the new error function $S[W_2] (x,t)$ is given by
\begin{equation}\label{ee2}\begin{split}
S[W_2] (x,t) &= S[W_1] (x,t) - \left({\la_0 \over d_0} \right)^{n-2} \left[ \Delta w + p W_0^{p-1} w \right] \eta \left( { |x-\xi| \over b \, d_0 }  \right) \\
&+ e_5 (x,t) + e_6 (x,t)
\end{split}
\end{equation}
where
\begin{equation}\begin{split} \nonumber
e_5 (x,t) & = \left({\la_0 \over d_0} \right)^{n-2} W_t ,\\
e_6 (x,t)&= \left({\la_0 \over d_0} \right)^{n-2} \left[ \Delta \bar w + p W_0^{p-1} \bar w \right]\eta \left({ |x-\xi| \over b \, d_0 } \right) \\
&+\left({\la_0 \over d_0} \right)^{n-2} \, \left[ 2 \nabla (w -\bar w) \nabla \left( \eta ( { |x-\xi| \over b \, d_0 }) \right) + (w-\bar w) \Delta
\left( \eta ( { |x-\xi| \over b \, d_0 }) \right) \right] \\
&+ {d \over dt} \left[ \left({\la_0 \over d_0} \right)^{n-2} \right] W -\left({\la_0 \over d_0} \right)^{n-2} {1\over x_1} {\partial W \over \partial x_1} \\
&+ \left( W_1 - \left({\la_0 \over d_0} \right)^{n-2} W \right)^p - W_1^p + p \left({\la_0 \over d_0} \right)^{n-2} W_1^{p-1} W \\
&+ p \left({\la_0 \over d_0} \right)^{n-2} ( W_0^{p-1} - W_1^{p-1} ) W.
\end{split}
\end{equation}
Observe that the function $e_6$ depends only on $\la_1 $ and $d_1$, but it does not depend on $\dot \la_1$, nor on $\dot d_1$. On the other hand, $e_5$ depends on all $\la_1 $, $d_1$, $\dot \la_1$ and $\dot d_1$.

\medskip
Next Lemma contains a description of the error function $S[W_2] (x,t )$ in a region close to $\xi$. An immediate comparison between the expression of $S[W_1]$ in \eqref{app1} and the one of $S[W_2]$ in \eqref{app2} shows that the new approximate solution $W_2$ corrects the error $E_0$ in this region.

\medskip

\begin{lemma}\label{lemma2}
Assume the functions $\la_1$ and $d_1$ satisfy \eqref{n1}, and that $T$ is small.
Let $\delta >0$ be a small fixed number and $y = {x-\xi \over \la}$. In the region $|x-\xi |< \delta d$, the error of approximation $S[W_2] (x,t)$ is as
follows
\begin{equation}\label{app2}\begin{split}
\la^{n+2 \over 2} S[W_2] (x,t)& =  E_{2,\la} [\la_1 , \dot \la_1 , d_1] (y,t) \\
&+ E_{2,d} [d_1 , \dot d_1 , \la_1 ] (y,t) + E[\la_1 , \dot \la_1 , d_1 , \dot d_1] (y,t)
\end{split}
\end{equation}
where
\begin{equation*}\begin{split}
E_{2,\la}  [\la_1 , \dot \la_1 , d_1] (y,t)& =(\la \dot \la_1 + \dot \la_0 \la_1 ) \, Z_0 (y) \\
&- (\la \dot \la_1 + \dot \la_0 \la_1 )  \,  \left({\la_0 \over d_0} \right)^{n-2} \, \left( {n-2 \over 2} h(y) + \nabla h(y) \cdot y \right) \\
&-{p \, (n-2) \, \alpha_n \over 2^{n-2}}  \left({\la_0 \over d_0} \right)^{n-2} \left[ {\la_1 \over \la_0} - {d_1 \over d_0}  \right] \left[ 1+ q_1 ({\la_1 \over \la_0}, {d_1 \over d_0} ) \right] U^{p-1} (y)\\
E_{2,d} [d_1 , \dot d_1 , \la_1] (y,t)& = \la \left[ \dot d_1 - {d_0 + d_1 + \la y_1 \over 1+ d + \la y_1} \right] Z_1 (y)\\
&- \la \dot d_1 \,  \left({\la_0 \over d_0} \right)^{n-2} \, {\partial h \over \partial y_1} (y) \\
&+  {p \, (n-2) \, \alpha_n \over 2^{n-1}} \left({\la \over d} \right)^{n-1} U^{p-1} (y) \, y_1\\
E[\la_1 , \dot \la_1 , d_1 , \dot d_1] (y,t)&=\la \dot d_1 \left( {\la_0 \over d_0} \right)^{n-1} f(y, {\la_1 \over \la_0 }  , {d_1 \over d_0} ) - \la \dot \la_1 \left( {\la_0 \over d_0} \right)^{n-2} f(y, {\la_1 \over \la_0 }  , {d_1 \over d_0} )\\
&+ \left( {\la_0 \over d_0} \right)^{n+2} f(y, {\la_1 \over \la_0 }  , {d_1 \over d_0} ).
\end{split}
\end{equation*}
Here $f = f(y, {\la_1 \over \la_0 }  , {d_1 \over d_0} )$ denotes a generic function, which is smooth and bounded for $y$ in the considered region, and for $\la_1 $ and $d_1$ satisfying \eqref{n1}, whose expression changes from line to line. With $q_1$ we denote a generic smooth real function, with the property that $q_1 (0,0) = 0$, and $\nabla q_1 (0,0) \not= 0$.

\end{lemma}

\medskip
\begin{proof}
Let $y = {x-\xi \over \la}$ and consider the region $|y|<\delta {d \over \la}$, for some fixed number $\delta$.
We take $\delta$ to be small enough
 so that $\eta \left(  { |x-\xi| \over b \, d_0 }  \right)$ is identically equal to $1$ in the region we are considering, at any time $t$. This is possible thanks to \eqref{n1}, taking $T$ smaller if necessary.
The function $e_5 (x,t)$ defined in \eqref{ee2}, is explicitly given by
\begin{equation*}\begin{split}
e_5 (x,t)&= -\left({\la_0 \over d_0} \right)^{n-2} \la^{-{n\over 2}} \left[ \dot \la ({n-2 \over 2} h(y) + \nabla h (y) \cdot y ) + \dot d {\partial h
\over \partial y_1} (y) \right] \\
&+\left({\la_0 \over d_0} \right)^{n-2} \la^{-{n\over 2}} \left[ \dot \la ({n-2 \over 2} h(y+ 2{d \over \la} {\bf e}_1) - \nabla h (+ 2{d \over \la} {\bf e}_1) \cdot (y+ 2{d \over \la} {\bf e}_1) ) \right]\\
&+ \left({\la_0 \over d_0} \right)^{n-2} \la^{-{n\over 2}}  \, \dot d \, {\partial h
\over \partial y_1} (y+ 2{d \over \la} {\bf e}_1) .
\end{split}
\end{equation*}
Taking advantage of the estimate \eqref{es1}, we can write
\begin{equation}\label{ll21}
\begin{split}
\la^{n+2 \over 2} e_5 (x,t) &= - \la \, \dot \la \, \left({\la_0 \over d_0} \right)^{n-2} \left({n-2 \over 2} h(y) + \nabla h (y) \cdot y \right) \\
&- \la \, \dot d \, \left({\la_0 \over d_0} \right)^{n-2} {\partial h
\over \partial y_1} (y) \\
&+ \la \, \dot \la \, \left({\la_0 \over d_0} \right)^{n} \, f(y, {\la_1 \over \la_0} , {d_1 \over d_0} )
+ \la \, \dot d \, \left({\la_0 \over d_0} \right)^{n+1} \, f(y, {\la_1 \over \la_0} , {d_1 \over d_0} )
\end{split}
\end{equation}
where $f = f(y, {\la_1 \over \la_0 }  , {d_1 \over d_0} )$ denotes a generic function, which is smooth and bounded for $y$ in the considered region, and for $\la_1 $ and $d_1$ satisfying \eqref{n1}.

\medskip
Next, we claim that
\begin{equation}\label{ll22}
\la^{n+2 \over 2} e_6 (x,t) = \left( {\la_0 \over d_0} \right)^{n+2} f(y, {\la_1 \over \la_0 }  , {d_1 \over d_0} ),
\end{equation}
for some $f$ as before. To check the validity of \eqref{ll22}, we start with the observation that
\begin{equation*}
\begin{split}
\left({\la_0 \over d_0} \right)^{n-2} &\left[ \Delta \bar w + p W_0^{p-1} \bar w \right] = \left({\la_0 \over d_0} \right)^{n-2} p \left[   W_0^{p-1} - \bar W_0^{p-1}   \right] \bar w \\
&+\left({\la_0 \over d_0} \right)^{n-2} \, \la^{-{n+2 \over 2}} \pi \left( {x-\hat \xi \over \la} \right).
\end{split}
\end{equation*}
Using again estimate \eqref{es1}, and a Taylor expansion, we get that
$$
\la^{n+2 \over 2} \left({\la_0 \over d_0} \right)^{n-2} \left[ \Delta \bar w + p W_0^{p-1} \bar w \right]= \left( {\la_0 \over d_0} \right)^{n+2} f(y, {\la_1 \over \la_0 }  , {d_1 \over d_0} ).
$$
Observe now that
$$
\la^{n+2 \over 2} \, {d \over dt} \left[ \left({\la_0 \over d_0} \right)^{n-2} \right] W = \la^2 (T-t)^{2 \over n-4} f(y, {\la_1 \over \la_0 }  , {d_1 \over d_0} ),
$$
while
$$
\la^{n+2 \over 2} \left({\la_0 \over d_0} \right)^{n-2} {1\over x_1} {\partial W \over \partial x_1}=
\la \left({\la_0 \over d_0} \right)^{n-2} f(y, {\la_1 \over \la_0 }  , {d_1 \over d_0} ).
$$
A direct computation thus gives that both terms
$$\la^{n+2 \over 2} \, {d \over dt} \left[ \left({\la_0 \over d_0} \right)^{n-2} \right] W
\quad {\mbox {and}} \quad
\la^{n+2 \over 2} \left({\la_0 \over d_0} \right)^{n-2} {1\over x_1} {\partial W \over \partial x_1}
$$ can be described as
$ \left( {\la_0 \over d_0} \right)^{n+2} f(y, {\la_1 \over \la_0 }  , {d_1 \over d_0} )$, for  $f = f(y, {\la_1 \over \la_0 }  , {d_1 \over d_0} )$ smooth and bounded for $y$ in the considered region, and for $\la_1 $ and $d_1$ satisfying \eqref{n1}.
Taylor expanding in $\left( W_1 - \left({\la_0 \over d_0} \right)^{n-2} W \right)^p - W_1^p + p \left({\la_0 \over d_0} \right)^{n-2} W_1^{p-1} W$, and using again \eqref{es1}, one gets a similar expression also for this  term. The last term in the definition of $e_6$ can be estimated in a similar way. Putting all the above information together, \eqref{ll22} is proven.

\medskip
The proof of expansion \eqref{ee2} thus directly follows from \eqref{ll21}, \eqref{ll22} combined with expansion \eqref{ee1} in Lemma \ref{lemma1} and the definition of $W(x,t)$ in \eqref{appro2}, together with \eqref{es2}.
\end{proof}

\medskip
At this point we write the error function $S[W_2]$ as in \eqref{gia}
\begin{align}\label{pey0}
S[W_2] (x,t) &=  \left[ \underbrace{e_1 + e_2  - \left({\la_0 \over d_0} \right)^{n-2} \left[ \Delta w + p W_0^{p-1} w \right] \eta \left(
 { |x-\xi| \over b \, d_0 } \right) }_{:=E_2} \right]  \eta_R (x,t) \nonumber \\
&+ \bar E_2 (x,t) ,
\end{align}
where $\eta_R$ is the cut off function introduced in \eqref{cutoff}.
We also fix the radius $R$ in terms of $T$. We  assume that
\begin{equation}\label{defT1}
R= T^{-\kappa}, \quad {\mbox {with}} \quad 0<\kappa< {1\over n-4}.
\end{equation}
Under this assumption, we have that $R \lesssim {d \over \la } (t)$, for all $t \in (0,T)$.

%Moreover, we assume that the radius $R$ satisfies
%\begin{equation}\label{defR}
%R >  T^{-{1\over 2 (n-4)}}  .
%\end{equation}
 The term $\bar E_2$ in \eqref{pey0} encodes the information of the error $S[W_2]$ regarding the lower order terms and the part of the main terms far away from the concentrating point $\xi$.
In the remaining part of the Section we describe the error function $\bar E_2$, and also
the approximation $W_2 (x,t)$ itself when evaluated in the boundary of $\DD$.

For later purpose, we need to estimate this part of the error, $\bar E_2$, in certain weighted $L^\infty$ norm.
Let $\alpha $ be a fixed number with $\alpha \in (0,1)$ and let $\sigma$ be the number fixed in \eqref{sigmadef}. For any smooth function $f=f(x,t)$, $x\in \DD$ and $t \in (0,T)$, we define the norm
\begin{equation}\begin{split} \label{defstarstar}
\| f \|_{**, \alpha, \sigma} := \inf \Biggl\{ M>0 \, : \, \la^{n+2 \over 2} & |f(x ,t) | \leq M \, \Biggl( \omega_{**,1} (y,t) + \omega_{**,2} (y,t) \Biggl)  \Biggl\} ,
\end{split}
\end{equation}
where, for $y= {x-\xi \over \la}$,
\begin{equation}
\begin{split} \label{defomegastarstar}
\omega_{**,1} (y,t) &= ({\la \over d} )^{\sigma} {1\over ({d \over \la})^{n-2} + |y|^{n-2}} {1\over (1+ |y|)^{2+\alpha}} \\
\omega_{**,2} (y,t) &= ({\la \over d} )^{n-2 +\sigma} {1\over (1+ |y|)^{n-3}} .
\end{split}
\end{equation}
 For any smooth function $g= g(x,t)$ defined in $\partial \DD \times (0,T)$, we introduce the norm on the boundary
\begin{equation}\label{normb}
\| g \|_{\partial \DD}:= \| \la^{n-2 \over 2} (t) \, \left( {\la_0 \over d_0} (t) \right)^{-n+2 -\sigma} \, (T-t)^{-{\alpha \over n-4}}  \,
 g(x,t) \|_{L^\infty (\partial \DD \times (0 , T) )}.
 \end{equation}

\medskip

\medskip
In the next Lemma, we describe the part of the error we called  $\bar E_2$ in the whole $\DD \times (0,T)$, and its Lipschitz dependence on $\la_1 $ and $d_1$. When needed, to emphasize the dependence of $\bar E_2$ on the parameter functions $\la_1$ and $d_1$ we use the notation
$$
\bar E_2 (x,t) = \bar E_2 [\la_1 , d_1] (x,t).
$$

\medskip

\begin{lemma}\label{lemma3}
Assume the functions $\la_1$ and $d_1$ satisfy \eqref{n1}. Let $\alpha \in (0,{1\over 2})$ be small and fixed, and $\sigma$ as in \eqref{sigmadef}. Let $T$ be small and $R$ satisfy \eqref{defT1}. Then there exists $\ve >0$, which depends only on $n$,
\begin{equation}\label{ee3}
\| \bar E_2 \|_{**, \alpha , \sigma} \lesssim T^\ve,
\quad
\| W_2 \|_{\partial \DD} \lesssim T^{1-{\sigma + \alpha \over n-4}} .
\end{equation}
Moreover, there exist  a positive number $\ve_1$, which depends only on $n$, such that, for any parameter functions $d_1$, and $\la_1^1$, $\la_1^2$ satisfying \eqref{n1}, one has
\begin{equation}\label{ll1}
\| \bar E_2 [\la_1^1, d_1] - \bar E_2 [\la_1^2 , d_1] \|_{**, \alpha , \sigma} \lesssim  T^{\ve_1}   \| \dot \la_1^1 - \dot \la_1^2 \|_{1+\sigma \over n-4},
\end{equation}
and
\begin{equation}\label{ll1b}
\| W_2 [\la_1^1, d_1] - W_2 [\la_1^2 , d_1] \|_{\partial \DD } \lesssim T^{\ve_1} \| \dot \la_1^1 - \dot \la_1^2 \|_{1+\sigma\over n-4},
\end{equation}
Also: for any parameter functions $\la_1$, and $d_1^1$, $d_1^2$ satisfying \eqref{n1}, one has
\begin{equation}\label{ll2}
\| \bar E_2 [\la_1, d_1^1] - \bar E_2 [\la_1 , d_1^2] \|_{**, \alpha , \sigma} \lesssim  T^{\ve_1}  \| \dot d_1^1 - \dot d_1^2 \|_{1+\sigma \over n-4}
\end{equation}
and
\begin{equation}\label{ll2b}
\| W_2 [\la_1, d_1^1] - W_2 [\la_1 , d_1^2] \|_{\partial \DD} \lesssim T^{\ve_1} \| \dot d_1^1 - \dot d_1^2 \|_{1+\sigma \over n-4}
\end{equation}
\end{lemma}

\medskip

\begin{proof}
We start the analysis of the second estimate in \eqref{ee3}.
If $x \in \partial \DD$, then $\eta \left( { |x-\xi| \over b \, d_0 } \right) \equiv 0$, thanks to the choice of $b>0$, see \eqref{appro2}. A Taylor expansion gives
$$
W_2 (x,t) = \la^{n-2 \over 2}  g(x, t) \left(1+ O(\la^2) \right), \quad
g(x,t) :={1\over |x-\xi|^{n-2} }- {1\over |x-\hat \xi |^{n-2}} ,
$$
uniformly for $x \in \partial \DD$. We claim that
$$
g(x,t) = O( {1\over d^{n-3}} ), \quad {\mbox {uniformly on }} \quad \partial \DD.
$$
This is certainly true if $x$ is a point of the boundary, far from $p:=(1,\bar 0)$, say if $d(x,p)> r_0 \sqrt{d}$, for some constant $r_0$.
Observe now that if  $x \in \partial \DD$ is such that $d(x, (1,\bar 0) ) \leq {\sqrt d}$, then we can assume that $x = (\phi (\bar x) , \bar x)$,
with $\phi$ a smooth function so that $\phi (\bar 0) = 1$, $\nabla \phi (\bar 0 ) =0$, and $D^2 \phi (\bar 0 ) \not= 0$.
Thus, for $x$ in this region, a simple Taylor expansion gives the existence of a constant $c$ so that
$
|g(x,t) | \leq c {d^2 \over |x-\xi |^n} , \quad {\mbox {for}} \quad x\not= p, \quad g(p,t ) =0.
$
We can conclude that, for any $x \in \partial \DD$, one has
$$
\left| W_2 (x,t) \right| \leq c {1\over \la^{n-2 \over 2}} \, \left({\la \over d} \right)^{n-2} \, d,
$$
so that $ \| W_2 \|_{\partial \DD} \lesssim T^{{n-4-\alpha -\sigma\over n-4}}.$
The second estimate in \eqref{ee3} readily follows.

\medskip
Let us check \eqref{ll1b}. Let $d_1$, and $\la_1^1$, $\la_1^2$ satisfy \eqref{n1}. For any $x \in \partial \DD$, a Taylor expansion gives
\begin{equation}\begin{split}\label{uffa}
\left| W_2 [\la_1^1 , d_1] (x,t)- W_2 [\la_1^2 , d_1] (x,t) \right| &\leq \la^{-{n \over 2}} \left| Z_0 ({x-\xi \over  \la} ) -
Z_0 ({x-\hat \xi \over  \la} ) \right| |\la_1^1 - \la_1^2 |
%\\
%&+  \la^{-{n \over 2}} \left| \pi_0 ({x-\xi \over  \la} ) -
%\pi_0 ({x-\hat \xi \over \la} ) \right| |\la_1^1 - \la_1^2 |
\end{split}
\end{equation}
for some $ \la = \la_0 + \bar \la$, with $\bar \la$ satisfying \eqref{n1}. Arguing as before, and using \eqref{n1},  we get
\begin{equation*}\begin{split}
\la^{-{n \over 2}}&\left| Z_0 ({x-\xi \over  \la} ) -
Z_0 ({x-\hat \xi \over  \la} ) \right| |\la_1^1 - \la_1^2 | \lesssim \la^{-{n\over 2}} \left({\la_0 \over d_0}\right)^{n-2} d (T-t)^{1+ {1+\sigma \over n-4}} \| \dot \la_1^1 - \dot \la_1^2 \|_{1+\sigma \over n-4} \\
&\lesssim \la^{-{n-2 \over 2}} \, \left( {\la_0 \over d_0} \right)^{n-2+\sigma } \, T^{n-4-\alpha -\sigma \over n-4} \,  \| \dot \la_1^1 - \dot \la_1^2 \|_{1+\sigma \over n-4}.
\end{split}
\end{equation*}
This proves  \eqref{ll1b}. In a similar way, one can show the validity of \eqref{ll2b}.

Let us show the validity of the first estimate in \eqref{ee3}. We write $\bar E_2$ explicitly
\begin{equation}\label{rr1}\begin{split}
\bar E_2 &= \underbrace{\left[ e_1 + e_2  - \left({\la_0 \over d_0} \right)^{n-2} \left[ \Delta w + p W_0^{p-1} w \right] \eta \left(
 { |x-\xi| \over b \, d_0 } \right) \right] (1- \eta_R (x,t)) }_{:= e_{out}}\\
&+ \sum_{j=3}^6 e_j (x,t).
\end{split}
\end{equation}
We refer to formulas \eqref{ee1} and \eqref{ee2} for the definition of $e_j$, $j=1, \ldots , 6$.

We start analyzing $e_{out}$. From Lemma \ref{lemma2} and \eqref{es2}, we get that for $R\lesssim |{x-\xi \over \la} | \lesssim {d\over \la}$, for some $a>0$,
$$
\la^{n+2 \over 2} |e_{out} (x,t)| \leq T^a \left( \omega_{**,1} + \omega_{**,2} \right).
$$
Let us then consider the region ${d\over \la} \lesssim |{x-\xi \over \la} |$.
Using that $\dot d_0 +1=0$, we have
$$
 \la^{{n+2 \over 2}} |e_1 | (1- \eta_R) \lesssim  {\la \dot d_1 \over 1+ |y|^{n-1}} (1-\eta_R) \lesssim  \omega_{**,2} (y,t)  R^{-2}.
$$

Writing $e_2 = \la^{-{n\over 2}} \dot \la Z_0 \left( {x-\xi \over \la } \right) - p W_0^{p-1} \bar W_0 =e_{21} + e_{22}$, we have for ${d\over \la} \lesssim |{x-\xi \over \la} |$, and some $a>0$,
$$
 \la^{{n+2 \over 2}} |e_{21}| (1-\eta_R) \lesssim    {\la \dot \la \over 1+ |y|^{n-2}} (1-\eta_R) \lesssim  \omega_{**,2} (y,t) T^a
$$
and
$$
 \la^{{n+2 \over 2}} |e_{22}| (1-\eta_R )\lesssim   {1\over 1+ |y|^4} {1\over ({d\over \la})^{n-2} + |y|^{n-2} } \lesssim
\omega_{**,1} (y,t) T^a.
$$
Moreover
\begin{equation*}\begin{split}
 \la^{{n+2 \over 2}} \left| \left({\la_0 \over d_0} \right)^{n-2} \left[ \Delta w + p W_0^{p-1} w \right] \eta \left(
 { |x-\xi| \over b \, d_0 } \right) \right| (1-\eta_R)  \lesssim
\omega_{**,1} (y,t) R^{-1}.
\end{split}
\end{equation*}
We thus conclude that, for some $\ve>0$,
$$
\| e_{out} \|_{**} \lesssim T^\ve.
$$
Referring to formula \eqref{ee1} for their definitions, we proceed with the estimate of $e_3$ and $e_4$. We have
\begin{equation*}
\begin{split}
 \la^{{n+2 \over 2}} |e_3| &\lesssim  {\la \over 1+ ({d\over \la})^{n-2} + |y|^{n-2}} \lesssim
({\la \over d})^{n-2} {1\over 1+ |y|^{n-3}} ({\la \over d} )^2 (T-t)^{-{1\over 4}} \\
&\lesssim \omega_{**,2} (y,t) T^{1+ {3\over n-4}}
\end{split}
\end{equation*}
and
\begin{align*}
|e_4|& \lesssim |\bar W_0 |^p \lesssim \la^{-{n+2 \over 2}} {d \over ({d\over \la})^{n+2} + |y|^{n+2} }
\lesssim \la^{-{n+2 \over 2}} \omega_{**,1} (y,t) T^\ve, \quad \ve>0.
\end{align*}
Referring to \eqref{ee2}, we next estimate $e_5$ and $e_6$. We have
\begin{align*}
|e_5|&\lesssim \la^{-{n+2 \over 2}} ({\la \over d} )^{2 (n-2)} {1\over 1+ |y|^2} \eta \left({ |x-\xi| \over b \, d_0 } \right) \lesssim
\la^{-{n+2 \over 2}} \omega_{**,2} (y,t) T^3,
\end{align*}
and
\begin{align*}
 |e_6| &\lesssim \la^{-{n+2 \over 2} } ({\la \over d} )^{n-2} {d\over 1+ |y|^4} \eta \left({ |x-\xi| \over b \, d_0 } \right) +
 \la^{-{n+2 \over 2} } ({\la \over d} )^{n-2} {\la^2 \over 1+ |y|^{n-2}} + |W|^p \\
 &\lesssim \la^{-{n+2 \over 2} } \left( \omega_{**,1} (y,t) + \omega_{**,2} (y,t) \right) T.
 \end{align*}
We conclude that there exists $\ve >0$ so that
$$
\| \sum_{j=3}^6 e_j \|_{**, \alpha , \sigma } \lesssim T^\ve .
$$

\medskip
The Lipschitz dependence of $\bar E_2$ with respect to the topology of the set to which $\la_1$ and $d_1$ belong, as stated in \eqref{ll1} and \eqref{ll2}, follows from the analysis of each one of the terms of $\bar E_2$ in \eqref{rr1}. One has to study them both in a region relative close to $\xi$, where one takes advantage of the results contained in Lemma \ref{lemma2}, and in a region far from $\xi $, where the explicit expressions collected in \eqref{ee1} and \eqref{ee2} are of use.
\end{proof}

\section{Solving the outer problem}\label{outer}

This section is devoted to solve in $ \psi = \psi (x,t)$ the {\it outer problem} \eqref{sout} in the form of a non linear non local operator
$$
\psi(x,t)  = \Psi [ \la_1 , d_1 , \phi ]\, (x,t)
$$
of the parameter functions $\la_1 $ and $d_1$ satisfying the bounds \eqref{n1}, and of the function $\phi$ defined in \eqref{defw}-\eqref{defPhi}
and chosen in the following range.

\medskip
Let $a $ be a fixed number with $a \in (0,1)$ and $a>\alpha$, where $\alpha$ has been fixed in the definitions of the norms
$\| \cdot \|_{**,\alpha, \sigma}$ and $\| \cdot \|_{\partial \DD}$ in \eqref{defstarstar} and \eqref{normb} respectively, with $\sigma $ defined in \eqref{sigmadef}.
Let $\beta $ be a positive number, which depends on $n$ and $\sigma$, to be fixed later on. We assume that $\phi$ in \eqref{defw}-\eqref{defPhi} satisfies the following bound
\begin{equation} \label{bphi}
\| \phi \|_{in} := \| \phi \|_{a, \sigma}+ \| (1+ |y|) \nabla \phi \|_{a, \sigma} \leq T^\beta
\end{equation}
where $\| \phi \|_{a,\sigma}$ is the least number $K$ with
$$
 |\phi (y,t) | \leq K \, {R^{n+1-a} \over 1+ |y|^{n+1} } \, \left({\la \over d}\right)^{n-2 +\sigma}  , \quad {\mbox {for}} \quad  y \in B(0,2R).
$$

\medskip
For parameter functions $\la_1$, $d_1$ satisfying \eqref{n1}, and for functions $\phi$ satisfying \eqref{bphi}, we find a solution $\psi$ to the initial value Problem
\begin{equation}\label{sout1}\begin{split}
\psi_t &= \Delta \psi + {1\over x_1} {\partial \psi \over \partial x_1} + V \psi  +\left( ( \Delta -{\partial \over \partial t} ) \eta_R \right) \, \Phi + 2\nabla \Phi \nabla \eta_R   + {1\over x_1} {\partial \eta_R \over \partial x_1}  \Phi\\
&+p \left[ W_2^{p-1} - [ \la_0^{-{n-2 \over 2}} U ({x-\xi \over \la_0 } ) ]^{p-1} \right] \eta_{R'} \eta_R \Phi\\
&+ N [{\bf w} ] + \bar E_2 \quad {\mbox {in}} \quad \DD \times (0,T)\\
\psi &= -W_2 , \quad {\mbox {on}} \quad \partial \DD \times (0,T), \quad \psi = \psi_0 , \quad {\mbox {in}} \quad \DD\times \{ t=0 \}.
\end{split}
\end{equation}

For any smooth function $f=f(x,t)$, $x\in \DD$ and $t \in (0,T)$, we define the norm
\begin{equation}\begin{split} \label{defstarnorm}
\| f \|_{*, \alpha, \sigma} := \inf \Biggl\{ M>0 \, : \, \la^{n-2 \over 2} & |f(x ,t) | \leq M \, \Biggl( \omega_{*,1} (y,t) + \omega_{*,2} (y,t) \Biggl)  \Biggl\} ,
\end{split}
\end{equation}
where, for $y= {x-\xi \over \la}$,
\begin{equation}
\begin{split} \label{defomegastar}
\omega_{*,1} (y,t) &= \left({\la \over d} \right)^{\sigma}  {1\over ({d \over \la})^{n-2} + |y|^{n-2}} \left( { 1 \over 1+ |y|^{\alpha}} + {\la^\alpha \over (T-t)^{\alpha \over 2} } \right) \\
\omega_{*,2} (y,t) &= \left({\la \over d} \right)^{n-2+\sigma} \left( {1\over 1+ |y|^{n-5}} + {\la^{n-5} \over (T-t)^{n-5 \over 2}} \right).
\end{split}
\end{equation}

The solution to \eqref{sout1} will have $\| \cdot \|_{*,  \alpha, \sigma}$-norm bounded, for any small and smooth initial condition $\psi_0$.

\medskip
We have the validity of

\begin{prop}\label{probesterno}
Assume that the parameters $ \lambda_1$ and $d_1$ satisfy \eqref{n1},  and the
function $\phi$ satisfies the constraint
\eqref{bphi}.  Assume that  $T$ is small and that the radius $R$ defined in \eqref{defT1} satisfies the stronger condition
\begin{equation}\label{defRT}
R^{p (n+1) } T^{4\over n-4} <1.
\end{equation}
Let $\ve >0$ be the number introduced in Lemma \ref{lemma3}. Then there exists $\bar \ve \in (0,\ve)$ so that, for any
 $\psi_0 \in C_0^2 (\bar \DD )$ with
\be\nonumber  % \label{psi0rest}
\| \psi_0 \|_{L^\infty (\bar \DD) } + \| \nabla \psi_0 \|_{L^\infty (\bar \DD) }
\lesssim   T^{\bar \ve}  ,
\ee
Problem \eqref{sout1}
has a unique solution $\psi= \Psi(\la_1 , d_1 , \phi)$. This solution satisfies the following estimates: for $y = {x-\xi \over \la}$,
\be \label{leuco1}
\la^{n-2 \over 2} \, |\psi(x,t)| \ \lesssim \ T^{\bar \ve} \left( \omega_{*,1} (y,t) + \omega_{*,2} (y,t) \right)
,
\ee
and, for $|y| <R$,
\be\label{leuco11}
\la^{n-2 \over 2} \,
|\nabla_x \psi(x,t)| \ \lesssim \ T^{\bar \ve} \left( {\la \over d} \right)^{n-2+\sigma}  \,   {\la^{-1} \over {|y|^{\alpha +1} + 1 }}  .
\ee

\end{prop}

\medskip
\noindent
To prove this result,  we shall estimate, for  given functions $f(x,t)$, $g(x,t)$, $h(x)$ the unique solution  of the linear problem
\be\label{hh1}
\pp_t \psi =  \Delta \psi + {1\over x_1} {\partial \psi \over \partial x_1} + V \psi \ + f(x,t)  \inn \DD \times (0,T),
\ee
$$
\qquad \psi = g   \onn  \pp\DD \times (0,T), \quad \psi(\cdot, 0) = h,
$$
where the function $V $ is defined in \eqref{defV}.

\begin{lemma}\label{psistar1}
Let $\psi = \psi[f,g,h] $ be the unique solution of Problem $\equ{hh1}$.

{\it Part (i).}
If $f(x,t)$ satisfies, for all $(x,t)$,
$$
\la^{n+2 \over 2} |f(x,t) | \leq \| f \|_{**,1} \, \omega_{**,1} (y,t), \quad y={x-\xi \over \la}
$$
with $\| f \|_{**,1}$ finite, then
\be \label{potoo0}
 \la^{n-2 \over 2}  |\psi(x,t)| \ \lesssim\   \left(  \|f\|_{**,1 }  +    \| h \|_{L^\infty (\DD ) }+ \| g \|_{\partial \DD }
 \right)   \omega_{1,*} (y,t)  .
\ee
We refer to \eqref{defomegastarstar}, \eqref{defomegastar} and \eqref{normb} for the definition of $\omega_{**,1}$, $\omega_{*,1}$ and
 $\| \cdot \|_{L^\infty (\DD ) }$ respectively.
 Moreover, we have the following local estimate on the gradient
\begin{align} \label{potoo1}
 \la^{n-2 \over 2} |\nn_x\psi(x,t)|& \ \lesssim\   \left(  \|f\|_{**,1 }  +   \| h \|_{L^\infty (\DD ) }+ \| g \|_{\partial \DD }
 \right)   \left({\la \over d} \right)^{n-2 +\sigma} {\la^{-1} \over  1+ \left |y\right |^{\alpha+1}   }  ,
\end{align}
for $|y |\leq R$.

\medskip
{\it Part (ii).}
If $f(x,t)$ satisfies, for all $(x,t)$,
$$
\la^{n+2 \over 2} |f(x,t) | \leq \| f \|_{**,2} \, \omega_{**,2} (y,t), \quad y={x-\xi \over \la}
$$
with $\| f \|_{**,2}$ finite, then
\be \label{potoo0nn}
 \la^{n-2 \over 2}  |\psi(x,t)| \ \lesssim\   \left(  \|f\|_{**,2 }  +    \| h \|_{L^\infty (\DD ) }+ \| g \|_{\partial \DD }
 \right)   \omega_{*,2} (y,t)  .
\ee
We refer to \eqref{defomegastarstar}, \eqref{defomegastar} and \eqref{normb} for the definition of $\omega_{**,2}$, $\omega_{*,2}$ and
 $\| \cdot \|_{L^\infty (\DD ) }$ respectively.
 Moreover, we have the following local estimate on the gradient
\begin{align} \label{potoo1nn}
 \la^{n-2 \over 2} |\nn_x\psi(x,t)|& \ \lesssim\   \left(  \|f\|_{**,1 }  +   \| h \|_{L^\infty (\DD ) }+ \| g \|_{\partial \DD }
 \right)   \left({\la \over d} \right)^{n-2 +\sigma} {\la^{-1} \over  1+ \left |y\right |^{n-4}   }  ,
\end{align}
for $|y |\leq R$.

\end{lemma}

\proof
To prove the results for both parts (i) and (ii), we construct a super solution for \eqref{hh1}.  We will do it in full details to prove Part (i) of the Lemma.  To this end,
let $q(|z|) = \frac 1{1+ |z|^{2+\alpha}} $ and let $Q(|z|)$ be the radial positive solution of
$$
\Delta Q   + 4 M q = 0\inn \R^n
, \quad
{\mbox {given by}} \quad
Q(r) =   4 M \int_r^\infty \frac {d\rho}{ \rho^{n-1}} \int_0^\rho   q(s)s^{n-1} ds,
$$
for a constant $M$ to be fixed later on.
Observe that
$
Q(z) \sim  \frac 1{1+ |z|^{\alpha } }\inn \R^n.
$
One has
$$
\Delta Q   +  \frac {\delta} { 1 + |z|^2}Q +   3 M q  \le  0\inn \R^n
$$
provided $\delta $ is small enough.
Define
$
\bar Q  (x) :=  Q\left (\frac{x-\xi}{\la}\right )
$ and $\bar q (x) := \frac 1{\la^2 }  q\left (\frac{x-\xi}{\la}\right )$.
For a possibly smaller $\delta$, one has
$$
\Delta_x \bar Q   +    \la^{-2 }  \,  \frac {\delta} { 1 +  \left |\frac{x-\xi}{\la}\right |^2} \    \bar Q \, + \  3 M\, \bar q \ \le\  0\inn \R^n.
$$
Observe now that
$$
| V (x,t) | \leq A   \,  {\la^{-2} R^{-2} \over 1+ |y|^2} ,
$$
for some constant $A$ independent of $t$ and $T$, as a direct consequence of the definition of  $V$ given in \eqref{defV}, and the bounds \eqref{n1} on the parameter functions $\la_1 $ and $d_1$. Moreover,
$$
|{1\over x_1} {\partial \bar Q \over \partial x_1} (x,t)| \lesssim {\la^{-1} \over |(\xi + \la y)_1 |} {1\over 1+ |y|^{1+\alpha}} \lesssim B {\la^{-2} \over 1+ |y|^{2+\alpha} },
$$
for a constant $B$ depending on the diameter of $\DD$.
From the above estimates, choosing $M $ large in terms of the diameter of $\DD$, we obtain
$$
\Delta \bar Q + {1\over x_1} {\partial \bar Q \over \partial x_1} + V \bar Q + 2M  \bar q \leq 0
$$
thanks to the fact that $R$ is large.

Define $Q_1 (x,t) = K_1 \left({\la \over d} \right)^{\sigma}  {\la^{-{n-2 \over 2}} \over ({d \over \la})^{n-2} + |y|^{n-2}}  \bar Q (x)$, for $y={x-\xi \over \la}$ and $K_1$ a constant.
Let $\pi (y,t) = \left({\la \over d} \right)^{\sigma}  {1 \over ({d \over \la})^{n-2} + |y|^{n-2}}$. A direct computation shows that
$$
|\nabla_x \bar Q \nabla_x \pi| + |\bar Q \Delta \pi | + |{1\over x_1} \bar Q \pp_{x_1} \pi| \leq c_1 K_1 {\la^{-2} \over 1+ |y|^{2+\alpha}} \pi,
$$
and
$$
|\pp_t Q_2 | \leq c_2 K_1 {\la^{-{n-2 \over 2}} (T-t ) \over 1+ |y|^\alpha} \, \pi (y,t)
$$
where $c_1$, $c_2$ are two positive constants that depend only on $n$. Possibly choosing $M$ larger if necessary, we find two positive constants $A$ and $B$ so that
%We thus have
%\begin{align*}
%\Delta Q_1 + {1\over x_1} {\partial  Q_1 \over \partial x_1} + V \bar Q_1 + f(x,t)& \leq
%{\la^{-{n+2 \over 2}} \over 1+ |y|^{2+ \alpha}} \pi (y,t) \left[ -2M K_1 +c K_1 + \|f \|_{**,1} \right]\\
%&\leq - M k_1 {\la^{-{n+2 \over 2}} \over 1+ |y|^{2+ \alpha}} \pi (y,t),
%\end{align*}
%choosing $M$ larger, if necessary, and
\begin{align*}
-\pp_t Q_1 &+ \Delta Q_1 + {1\over x_1} {\partial  Q_1 \over \partial x_1} + V \bar Q_1 + f(x,t)\\
& \leq \| f \|_{**,1} \la^{-{n-2 \over 2}} \pi (y,t) \, \left[ -A {\la^{-2} \over 1+ |y|^{2+\alpha}} + B {(T-t )^{-1} \over 1+ |y|^\alpha} \right]\\
&\leq \begin{cases} 0, & \text{if $|y| < {\sqrt{T-t} \over \la} $},\\ \tilde c_1  \la^{-{n-2 \over 2}} \pi (y,t) { (T-t)^{-1} \over 1+|y|^\alpha} , & \text{if $|y| > {\sqrt{T-t} \over \la}$}. \end{cases}
\end{align*}
Let $Q_2 (x,t ) = c \| f \|_{**,1} \la^{-{n-2 \over 2}} \pi (y,t) {\la^\alpha \over (T-t)^{\alpha \over 2}}$, and $\psi (x,t) = Q_1 (x,t) + Q_2 (x,t)$. We have
$$
-(\psi)_t + \Delta \psi+ {1\over x_1} \pp_{x_1} \psi + V\psi+ f(x,t) \leq 0.
$$
Moreover,  one has $|g(x,t)| \leq K \psi (x,t)$, for $(x,t) \in \partial \DD \times (0,T)$, and $ |h(x) | \leq K\psi (x,0)$, for $x\in \DD$, provided the constant $K>0$ is properly chosen. Thus $K \psi$ is a positive super solution for \eqref{hh1}.
Estimate \equ{potoo0} thus follows from parabolic comparison.

\medskip
To get the gradient estimate in \equ{potoo1} we scale around  $\xi$
letting
$$
 \psi(x, t ) :=   \ttt \psi\left ( \frac {x- \xi}{\la} , \tau(t) \right ), \quad z={x-\xi \over \la},
$$
where
$\dot \tau (t) =  \la(t)^{-2}$.
We choose $T$ small so that $\tau \ge 2$.
Then $\ttt \psi$ satisfies  for  $|z| \le \delta \la^{-1}$, with sufficiently small $\delta$,
$$
\pp_\tau \ttt \psi =         \Delta_z \ttt \psi  +   a(z,t) \cdot \nn_z \ttt \psi   +  b(z,t) \ttt\psi + \ttt f(z,\tau)
$$
where
$
\ttt f(z,\tau)  =   \la^2 f( \xi + \la  z , t(\tau) ) ,
$
and  $a(z,t)$ and $b(z,t)$ are uniformly small coefficients.
%given by
%$$
%a(z,t) := [ \la {\dot \la}  z  + {\dot \xi} {\la} ], \quad  b(z,t) =  V ( \xi+ \la z )   =  O( R^{-4}) (1+|z|)^{-4}   .
%$$
Our assumption in $f$ implies  that in the region $|z| <\delta \la^{-1}$
$$
\la^{n+2 \over 2} |\ttt f(z,\tau)|\ \lesssim \   ({\la \over d})^{n-2+\sigma} \frac { \|f\|_{**, 1}}{ 1 + |z|^{2+\alpha} }
$$
while we have already  established that for $|z| <R$
$$
\la^{n-2 \over 2} |\ttt \psi (z,\tau)| \lesssim  ({\la \over d})^{n-2+\sigma} \|f\|_{**,1 } \frac 1{ 1 + |z|^{\alpha} } .
$$
Let us now fix $0<\eta <1$. By standard parabolic estimates we get that for $\tau_1 \ge \tau(t_0) + 2$,
\begin{equation*}\begin{split}
 [ \nn_z  \ttt \psi( \tau_1  , \cdot ) ]_{\eta, B_{10} (0)}  & +    \|  \nn_z  \ttt \psi( \tau_1 , \cdot )  \|_ {L^\infty ( B_{10} (0))}\ \\
 &\lesssim\
  \|  \ttt \psi  \|_ {L^\infty ( B_{20} (0)) \times (\tau_1-1 , \tau_1 )}  + \|  \ttt f  \|_ {L^\infty ( B_{20} (0)) \times (\tau_1-1 , \tau_1 )} \\
& \lesssim  \la^{-{n-2 \over 2}} ({\la \over d})^{n-2+\sigma}  \|f\|_{**,1}  .
\end{split} \end{equation*}
provided that $\tau_1 \ge 2$.  Translating this estimate to the original variables $(x,t)$ we find  that for any $t \ge   c_n t_0$, for a suitable constant $c_n$,
 $$
 ( R\la )^{1+\eta} [ \nn_x \psi( t  , \cdot ) ]_{\eta, B_{ 10 R\la } (\xi)}   +     R\la \|  \nn_x\psi( t , \cdot )  \|_ {L^\infty (B_{ 10 R\la } (\xi))}\ \lesssim\
 \la^{-{n-2 \over 2}} ({\la \over d})^{n-2 +\sigma}  \|f\|_{**,1} .
 $$
 Using similar parabolic estimate up to the initial condition $\psi_0$  at $0$ for $\psi$ yields the validity of  the above estimate, and hence of \equ{potoo1}, for any $t\ge 0$. The proof of Part (i)  is complete.

 \medskip
 To obtain the result in Part (ii), we define
 $$
Q_1 (x,t ) = K_1 \la^{-{n-2 \over 2}} \left({\la \over d} \right)^{n-2+\sigma} \bar Q (x,t), \quad \bar Q (x,t) = Q \left( {x-\xi \over \la } \right),
$$
where $K_1$ is a positive constant and $Q$ is the solutions to $\Delta Q + {4M \over 1+ |y|^{n-3}} = 0$ that decays as $|Q(y) | \leq {4M\over 1+ |y|^{n-5}}$, with $M$ a fixed constant, defined as $Q(r) =   4 M \int_r^\infty \frac {d\rho}{ \rho^{n-1}} \int_0^\rho   {s^{n-1} \over 1+ s^{n-3}} ds$, $r=|y|$. Recall that $n\geq 6$.
One has
\begin{align*}
-\pp_t Q_1 &+ \Delta Q_1 + {1\over x_1} {\partial  Q_1 \over \partial x_1} + V \bar Q_1 + f(x,t)\\
& \leq \| f \|_{**,2} \la^{-{n-2 \over 2}} \left({\la \over d} \right)^{n-2+\sigma} \, \left[ -A {\la^{-2} \over 1+ |y|^{n-3}} + B {(T-t )^{-1} \over 1+ |y|^{n-5}} \right]\\
&\leq \begin{cases} 0, & \text{if $|y| < {\sqrt{T-t} \over \la} $},\\ \tilde c_1  \la^{-{n-2 \over 2}} \left({\la \over d} \right)^{n-2+\sigma} { (T-t)^{-1} \over 1+|y|^{n-5}} , & \text{if $|y| > {\sqrt{T-t} \over \la}$},\end{cases}
\end{align*}
for some positive constants $A$ and $B$.
Let $Q_2 (x,t ) = c \| f \|_{**,2} \la^{-{n-2 \over 2}} \left({\la \over d} \right)^{n-2+\sigma} {\la^{n-5} \over (T-t)^{n-5 \over 2}}$, and $\psi (x,t) = Q_1 (x,t) + Q_2 (x,t)$. We have
$$
-(\psi)_t + \Delta \psi+ {1\over x_1} \pp_{x_1} \psi + V\psi+ f(x,t) \leq 0.
$$
Moreover,  one has $|g(x,t)| \leq K \psi (x,t)$, for $(x,t) \in \partial \DD \times (0,T)$, and $ |h(x) | \leq K\psi (x,0)$, for $x\in \DD$, provided the constant $K>0$ is properly chosen. Thus $K \psi$ is a positive super solution for \eqref{hh1}.
Estimate \equ{potoo0nn} thus follows from parabolic comparison.
The proof of  \equ{potoo1nn} is similar to the corresponding one in Part (i).
 \qed

\medskip

We now give the
\begin{proof}[Proof of Proposition \ref{probesterno}]
Combining Part (i) and Part (ii) in Lemma \ref{psistar1}, one defines a linear bounded operator $S(f,g,h) = \psi$, which is the solution to \eqref{hh1} and the existence of a constant $c>0$ such that
\begin{equation}
\nonumber %\label{hh2}
\| S(f,g,h)\|_{*,\alpha, \sigma} \leq c \left(   \|f\|_{**,\alpha, \sigma }  + \| h \|_{L^\infty (\DD ) }+ \| g \|_{\partial \DD}
  \right).
\end{equation}
We refer to \eqref{defstarnorm} and \eqref{defstarstar} for the definition of $\| \cdot \|_{*,\alpha, \sigma} $ and $\| \cdot \|_{**,\alpha,\sigma} $.
We establish the existence of a solution $\psi$ to \eqref{sout1}, satisfying \eqref{leuco1}, as a fixed point for the Problem
\begin{equation}
\label{hh44}
\psi = {\bf S} (\psi), \quad {\bf S} (\psi ):= S(f,g,h),
\end{equation}
where
\begin{equation}\nonumber \begin{split}
f&= \left( \Delta -{\partial \over \partial t} \right) \eta_R \, \Phi + 2\nabla \Phi \nabla \eta_R + {1\over x_1} {\partial \eta_R \over \partial x_1}  \Phi \\
&+p \left[ W_2^{p-1} - [ \la^{-{n-2 \over 2}} U ({x-\xi \over \la } ) ]^{p-1} \right] \eta_{R'} \eta_R \Phi + N [{\bf w} ] + \bar E_2,\quad\\
g&=-W_2 , \quad h=\psi_0.
\end{split}
\end{equation}
Let $\ve >0$ be the constant in \eqref{ee3}, Lemma \ref{lemma3}. We claim that there exist $\bar \ve \in (0,\ve)$ and  a fixed point $\psi$ for \eqref{hh44} in the set
$$
B_M =\{ \psi \, : \, \| \psi \|_{*,\alpha , \sigma } \leq M \, T^{\bar \ve} \}, \quad {\mbox {for some}} \quad M>0,
$$
as a consequence of the Contraction Mapping Theorem. Indeed, by Lemma \ref{psistar1}, there exists a constant $c$ such that, for any $\psi \in B_M$
$$
\| {\bf S} (\psi )\|_{*,\alpha, \sigma } \leq c \left( \| f \|_{**,\alpha ,\sigma } + \| h \|_{L^\infty (\DD ) }+ \| g \|_{\partial \DD } \right)
$$
From  \eqref{ee3} in Lemma \ref{lemma3}, we get that $\| \bar E_2 \|_{** \alpha , \sigma} \lesssim T^\ve$ and  $\| g \|_{\partial \DD } \lesssim  T^{1-{\alpha +\sigma \over (n-4)}} $.
Thus the map ${\bf S}$ sends the set $B_M$ into $B_M$ provided that $\| f \|_{**,\alpha , \sigma  } \leq c T^{\bar \ve} $, for some positive constant $c$ independent of $t$ and $T$. This last inequality follows from the fact that
\begin{equation}\label{hh6nnn}\begin{split}
&\| \left( \Delta -{\partial \over \partial t} \right) \eta_R \, \Phi + 2\nabla \Phi \nabla \eta_R + {1\over x_1} {\partial \eta_R \over \partial x_1}  \Phi \| _{**,\alpha , \sigma  } \lesssim R^{\alpha - a} \| \phi \|_{in}  \\
& \|
p \left[ W_2^{p-1} - [ \la^{-{n-2 \over 2}} U ({x-\xi \over \la } ) ]^{p-1} \right] \eta_{R'} \eta_R \Phi
+ N [{\bf w} ] \| _{**,\alpha , \sigma } \lesssim R^{-a} \| \phi \|_{in}
\end{split}
\end{equation}
combined with our assumptions on $R$ in \eqref{defT1}, \eqref{defRT}.

\medskip
We also claim that there exists a positive number $b$, which depends on $n$, so that
\begin{equation}\label{ffi1}
\| {\bf S} (\psi_1 ) - {\bf S} (\psi_2 ) \| _{*,\alpha , \sigma} \lesssim T^b \| \psi_1 - \psi_2 \|_{*,\alpha , \sigma} .
\end{equation}
Thus, the map ${\bf S}$ is a contraction. This concludes the proof of the existence of $\psi$ solution to \eqref{sout1}, satisfying estimate \eqref{leuco1}. Estimate \eqref{leuco11} follows from Lemma \ref{psistar1} and estimate \eqref{potoo1}.

\medskip
The rest of this proof is devoted to establish the validity of \eqref{hh6nnn} and \eqref{ffi1}

\medskip
\medskip
We now prove \eqref{hh6nnn}. Recall that
$$
 \Phi (x,t) = \la^{-{n-2 \over 2}} \phi \left( {x-\xi \over \la } , t\right)
 $$
 We star with the first estimate. Since we are assuming the bound \eqref{bphi} in the {\it inner} function $\phi$, we  observe that
$$
\la^{n-2 \over 2} \, \left| (\Delta -{\partial \over \partial t} ) \eta_R \Phi \right| \lesssim \left(
|{\eta'' \over R^2 \la^2} |+ | {\eta' \over R\la} {\dot \la \over \la} | + |\eta' {\dot d \over \la } | \right) \| \phi \|_{in} \,  \left( {\la \over d} \right)^{n-2+\sigma} \, {1\over 1+ |y|^a},
$$
so that, in the region where it is not zero (that is $R \leq |y| \leq 2R$), we get
$$
\la^{n-2 \over 2} \left| (\Delta -{\partial \over \partial t} ) \eta_R \Phi \right| \lesssim  R^{-a+\alpha} \| \phi \|_{in} \,  \left( {\la \over d} \right)^{n-2+\sigma } \,  { \la^{-2} \over 1+ |y|^{2+\alpha}} .
$$
Analogous estimate holds for the term ${1\over x_1} {\partial \eta_R \over \partial x_1}  \Phi$. Similarly, one has
$$
\la^{n-2 \over 2} \left| 2 \nabla \Phi \nabla \eta_R \right| \lesssim |{\eta' \over \la R} | \| \phi \|_{in} \,  \left( {\la \over d} \right)^{n-2+\sigma } \,  { \la^{-1} \over 1+ |y|^{1+a}} ,
$$
so that, in the region where it is not zero, we get
$$
\la^{n-2 \over 2} \left| 2 \nabla \Phi \nabla \eta_R \right| \lesssim  R^{-a+\alpha} \| \phi \|_{in} \,  \left( {\la \over d} \right)^{n-2+\sigma } \,  { \la^{-2} \over 1+ |y|^{2+\alpha}} .
$$
We conclude that
\begin{equation}\label{uffa2}
\| \left( \Delta -{\partial \over \partial t} \right) \eta_R \, \Phi + 2\nabla \Phi \nabla \eta_R + {1\over x_1} {\partial \eta_R \over \partial x_1}  \Phi \| _{**,\alpha } \lesssim  R^{-a+\alpha} \| \phi \|_{in}.
\end{equation}
Since $a>\alpha$ and \eqref{bphi},  the first estimate in \eqref{hh6nnn} is proven.

\medskip
Next we consider the second estimate in \eqref{hh6nnn}. We have
\begin{equation*}\begin{split}
\la^{n-2 \over 2} \, &\left| p \left[ W_2^{p-1} - [ \la^{-{n-2 \over 2}} U ({x-\xi \over \la } ) ]^{p-1} \right] \eta_{R'} \eta_R \Phi
\right| \\
&
%\lesssim \la^{-2} U^{p-1} (y)  \,  \| \phi \|_{in}  \, \left({\la \over d} \right)^{n-2 } \,  {T^{\sigma' \over n-4} R^{n+1-a} \over 1+ |y|^{n+1}} |\eta_R|
\lesssim     \| \phi \|_{in}  \, \left({\la \over d} \right)^{n-2 +\sigma  }\,    {\la^{-2} \over 1+ |y|^{2+\alpha}} \eta_R  {\la \over d} \,  R^{n+1-a} \\
&\lesssim \| \phi \|_{in}  \, \left({\la \over d} \right)^{n-2 +\sigma  }\,    {\la^{-2} \over 1+ |y|^{2+\alpha}} \eta_R  \,  R^{-a}
\end{split}
\end{equation*}
We conclude that
\begin{equation}\label{uffa22}
\| \la^{n-2 \over 2} \, \left( p \left[ W_2^{p-1} - [ \la^{-{n-2 \over 2}} U ({x-\xi \over \la } ) ]^{p-1} \right] \eta_{R'} \eta_R \Phi
\right) \| _{**,\alpha ,\sigma } \lesssim  R^{-a} \| \phi \|_{in} .
\end{equation}
In order to estimate $N({\bf w} )$, we write
\begin{equation}\label{ffi2}
W_2 (x,t) = {1\over \la^{n-2 \over 2} } [U(y) + \rho (y)].
\end{equation}
From \eqref{s2}, we get
\begin{equation*}\begin{split}
\la^{n+2 \over 2} |N({\bf w} )| &(\la y + \xi ) \lesssim   (U+\rho + \la^{n-2 \over 2}[ \psi + \eta_R \Phi] )^p -
(U+\rho )^p \\
&- p (U+\rho)^{p-1} \la^{n-2 \over 2} [\psi + \eta_R \Phi] \lesssim  |\la^{n-2 \over 2} \psi |^p + |\la^{n-2 \over 2} \eta_R \Phi |^p  .
%\\
%&\lesssim \left({\la_0 \over d_0} \right)^{(n-2+\sigma)p } \left( {\| \psi \|_\alpha ^p \over 1+ |y|^{\alpha p} } + {\| \phi \|_{in}^p \over 1+ |y|^{ap} } \right)\\
%&\lesssim \left({\la_0 \over d_0} \right)^{(n-2+\sigma) (p-1) } \, \left( \| \psi \|_\alpha^{p-1} + \| \phi \|_{in}^{p-1} \right) \, \left({\la_0 \over %d_0} \right)^{n-2+\sigma } \, \left( {\la^{-2} \over 1+ |y|^{2+\alpha} } + {\la^\alpha \over (T-t)^{\alpha \over 2}} \right).
\end{split}
\end{equation*}
We have that
$$
| \la^{n-2 \over 2} \psi (x,t) |^p \lesssim \| \psi \|_{*, \alpha, \sigma}^p \left( |\omega_{*,1}|^p + |\omega_{*,2}|^p \right)(y,t).
$$
We have
\begin{align}\label{pey1}
|\omega_{*,1}|^p& \lesssim {1\over ({d\over \la})^{n+2 } + |y|^{n+2}} \left( {1 \over 1+ |y|^{p \alpha} } + {\la^{\alpha p} \over (T-t)^{\alpha p\over 2}}\right) \nonumber \\
&\lesssim \omega_{**,1} (y,t) {1+ |y|^{2+\alpha} \over ({d\over \la})^4 + |y|^4}  \, \left( {1 \over 1+ |y|^{p \alpha} } + {\la^{\alpha p} \over (T-t)^{\alpha p\over 2}}\right) \nonumber \\
&\lesssim   \omega_{**,1} (y,t) \begin{cases} ({\la \over d})^2 , & \text{if $|y| < {\sqrt{T-t} \over \la} $}, \nonumber \\
 (T-t)^{{n-2 \over n-4} - {\alpha \over n-2} ({n+2 \over 2} - 4{n-3 \over n-4} )} , & \text{if $|y| > {\sqrt{T-t} \over \la}$},\end{cases} \nonumber \\
 &\lesssim T^b \omega_{**,1} (y,t)
\end{align}
and
\begin{align}\label{pey2}
|\omega_{*,2}|^p& \lesssim \left({\la \over d} \right)^{n+2} \left( {1 \over 1+ |y|^{(n-5) p} } + {\la^{(n-5) p} \over (T-t)^{(n-5) p\over 2}}\right) \nonumber \\
&\lesssim \omega_{**,2} (y,t) \left({\la \over d} \right)^{4-\sigma } (1+ |y|^2) \left( {1 \over 1+ |y|^{(n-5) p} } + {\la^{(n-5) p} \over (T-t)^{(n-5) p\over 2}}\right)\nonumber\\
&\lesssim   \omega_{**,2} (y,t) \begin{cases} ({\la \over d})^{4-\sigma} {1\over (1+ |y|)^{n-7 + 4 {n-5 \over n-2}}}  , & \text{if $|y| < {\sqrt{T-t} \over \la} $}, \nonumber \\
 (T-t)^{{n-3 \over 2} +{1 \over n-4} (1-\sigma + 4{n-5 \over n-2} )} , & \text{if $|y| > {\sqrt{T-t} \over \la}$},\end{cases} \nonumber \\
 &\lesssim T^b \omega_{**,2} (y,t)
\end{align}
for some $b >0$ which depends on $n$. From \eqref{pey1} and \eqref{pey2}, we get
$$
| \la^{n-2 \over 2} \psi (x,t) | \lesssim \| \psi \|_{*, \alpha, \sigma}^p  T^b \left( |\omega_{**,1}| + |\omega_{**,2}| \right)(y,t).
$$
Furthermore, since $\phi$ satisfies the bound \eqref{bphi}, we have
\begin{align}\label{pey3}
|\la^{n-2 \over 2} \eta_R \Phi |^p &\lesssim |\eta_R \phi |^p \lesssim \eta_R^p \left({\la \over d} \right)^{n+2 +p\sigma } \| \phi \|_{in}^p {R^{p(n+1-a)} \over (1+ |y|^{n+1} )^p}   \nonumber \\
&\lesssim \lesssim \eta_R^p \left({\la \over d} \right)^{n-2+\sigma } {1\over 1+ |y|^{2+\alpha}} \| \phi \|_{in}^p T^{4 + (p-1) \sigma \over n-4} R^{(n+1)p} R^{-ap} \nonumber \\
&\lesssim \omega_{**,1} (y,t) \eta_R \| \phi \|_{in}^p T^{4+(p-1) \sigma \over n-4} R^{(n+1)p} R^{-ap}.
\end{align}
Thanks to \eqref{pey3} and \eqref{defRT}, we conclude that
$$
\la^{n+2 \over 2} |N({\bf w} )| (\la y + \xi ) \lesssim T^\ve \left( \omega_{**,1} (y,t) + \omega_{**,2} (y,t) \right) \, \left( \| \psi\|^p_{*,\alpha, \sigma} + \| \phi \|_{in}^p \right),
$$
for some $\ve >0$. This, together with \eqref{uffa22},  concludes the proof of the second estimate in \eqref{hh6nnn}.

\medskip
\medskip
We now prove \eqref{ffi1}.
Observe  that, for any pair of functions $\psi_1$, $\psi_2 \in B_M$, we have
$$
\| {\bf S} (\psi_1 ) - {\bf S}(\psi_2 ) \| _{*,\alpha , \sigma } \leq c \| N ({\bf w}_1) - N ({\bf w}_2) \|_{**,\alpha, \sigma  }
$$
since  $g$ and $h$, as defined in \eqref{hh3}, do not depend on $\psi$. Here we denote
$$
{\bf w}_j=  \psi_j (x,t) + \eta_R (x,t)  \Phi (x,t), \quad j=1,2.
$$
We refer to \eqref{s2} for the definition of $N$. Using again \eqref{pey1}-\eqref{pey2}, we can write
\begin{equation*}\begin{split}
\la^{n+2 \over 2} &\left| N({\bf w}_1 ) - N({\bf w}_2) \right| (\la y + \xi , t) \lesssim  \Biggl[ (U+\rho + \la^{n-2 \over 2} (\psi_1 + \eta_R \Phi ))^p \\
&- (U+\rho + \la^{n-2 \over 2} (\psi_2 + \eta_R \Phi ))^p - p (U+\rho )^{p-1} \la^{n-2 \over 2} |\psi_1 - \psi_2| \Biggl]\\
&\lesssim \, |\la^{n-2 \over 2} (\psi_1 - \psi_2)  |^p \lesssim \| \psi_1 - \psi_2 \|_{*, \alpha, \sigma}^p  T^b \left( |\omega_{**,1}| + |\omega_{**,2}| \right)(y,t),
\end{split}
\end{equation*}
for some $b >0$, dependent on $n$.
From here, we get the validity of \eqref{ffi1}, and this concludes the proof of the Proposition.

\end{proof}

\medskip
We further observe that the solution $\psi = \psi [\la_1 , d_1 , \phi ]$  to Problem \eqref{sout1} clearly depends  on the parameter functions $\la_1 $, $d_1$ and $\phi$. Next Proposition
clarifies that $\psi$ is Lipschitz with respect to $\la_1$, $d_1$ and $\phi$ and their respective topologies.

\medskip
\begin{lemma}\label{lemmapsi}
Assume the validity of the hypothesis in Proposition \ref{probesterno}. There exists a positive number $b>0$, which depends on $n$ and $\sigma$,  so that, for any $\la_1^1$, $\la_1^2$ satisfying \eqref{n1}, we have
\begin{equation}
\label{ll3}
\| \psi [\la_1^1 , d_1 , \phi ] - \psi [\la_1^2 , d_1 , \phi] \|_{*,\alpha, \sigma } \leq T^b \| \dot \la_1^1 - \dot \la_1^2 \|_{1+\sigma\over n-4},
\end{equation}
for any $d_1^1$, $d_1^2$ satisfying \eqref{n1},
\begin{equation}
\label{ll4}
\| \psi [\la_1 , d_1^1 , \phi ] - \psi [\la_1 , d_1^2 , \phi] \|_{*,\alpha , \sigma} \leq T^b \| \dot d_1^1 - \dot d_1^2 \|_{1+\sigma\over n-4},
\end{equation}
and, for any $\phi_1$, $\phi_2$ satisfying \eqref{bphi}
\begin{equation}
\label{ll5}
\| \psi [\la_1 , d_1 , \phi_1 ] - \psi [\la_1 , d_1 , \phi_2] \|_{*,\alpha , \sigma} \leq T^b \| \phi_1 - \phi_2 \|_{in},
\end{equation}

\end{lemma}

\begin{proof} Estimates \eqref{ll3} and \eqref{ll4} follows from the Lipshitz bound on the error function $\bar E_2$ contained
 in \eqref{ll1} and \eqref{ll2}, and from the Lipschitz bound on the value of $W_2$ on the boundary $\partial \DD$ as described in \eqref{ll1b} and \eqref{ll2b}. We leave the details to the reader.

 \medskip
 We shall prove \eqref{ll5}. As in the argument to show the first estimate in \eqref{hh6nnn}, we need to chose the number $a$ in the definition of the norm \eqref{defstarnorm} and the number $\alpha $ in the definition of the norm \eqref{defstarstar} so that $a >\alpha$.

Let $\la_1$ and $d_1$ be fixed. The Implicit Function Theorem ensures that $\phi \to \psi (\phi)$,
solution to Problem \eqref{sout1}, is a smooth map. Let us define $Z(x, t) = {\partial \psi \over \partial \phi} [\bar \phi] (x, t)$, for functions
$\bar \phi$ satisfying \eqref{bphi}. Then $Z$ solves
\begin{equation}\nonumber \begin{split}
Z_t &= \Delta Z + {1\over x_1} {\partial Z \over \partial x_1} + V Z + \left( ( \Delta -{\partial \over \partial t} ) \eta_R \right)\, \bar \Phi + 2\nabla \bar \Phi \nabla \eta_R + {1\over x_1} {\partial \eta_R \over \partial x_1}  \bar \Phi \\
&+p \left[ W_2^{p-1} - [ \la^{-{n-2 \over 2}} U ({x-\xi \over \la } ) ]^{p-1} \right] \eta_{R'} \eta_R \bar \Phi\\
&+ p \left( (W_2 + \psi + \eta_R \Phi )^{p-1} - W_2^{p-1} \right) \eta_R \bar \Phi   \quad {\mbox {in}} \quad \DD \times (0,T)\\
\psi &= 0 , \quad {\mbox {on}} \quad \partial \DD \times (0,T), \quad \psi = 0 , \quad {\mbox {in}} \quad \DD\times \{ t=0 \}.
\end{split}
\end{equation}
where
$$
\Phi (x,t) = \la^{-{n-2 \over 2}} \phi ({x-\xi \over \la } , t), \quad \bar \Phi (x,t) = \la^{-{n-2 \over 2}} \bar \phi ({x-\xi \over \la } , t).
$$
Arguing as in \eqref{uffa2} and \eqref{uffa22}, we get that
$$
\la^{n+2 \over 2} \left| \left( ( \Delta -{\partial \over \partial t} ) \eta_R \right)\, \bar \Phi + 2\nabla \bar \Phi \nabla \eta_R + {1\over x_1} {\partial \eta_R \over \partial x_1}  \bar \Phi \right| \lesssim R^{-a+\alpha} \| \bar \phi \|_{in} \left( \omega_{**,1} + \omega_{**,2} \right)
$$
and
$$
\la^{n+2 \over 2} \left| p \left[ W_2^{p-1} - [ \la^{-{n-2 \over 2}} U ({x-\xi \over \la } ) ]^{p-1} \right] \eta_{R'} \eta_R \bar \Phi \right| \lesssim R^{-a} \| \bar \phi \|_{in} \left( \omega_{**,1} + \omega_{**,2} \right).
$$
On the other hand, one has
\begin{align*}
\la^{n+2 \over 2} & \left| p \left( (W_2 + \psi + \eta_R \Phi )^{p-1} - W_2^{p-1} \right) \eta_R \bar \Phi \right| \lesssim |\la^{n-2 \over 2} \psi |^{p-1} \, |\la^{n-2 \over 2} \bar \phi | \eta_R \\
&+ |\la^{n-2 \over 2} \eta_R \phi  |^{p-1} \, |\la^{n-2 \over 2} \bar \phi | \eta_R = A+B,
\end{align*}
with
\begin{align*}
|A|&\lesssim \left({\la \over d} \right)^{n+2+p \sigma} {R^{n+1-a} \over 1+ |y|^{n+1}}  \| \psi \|_{*,\alpha , \sigma} \| \bar \phi \|_{in}\\
&\lesssim \left({\la \over d} \right)^{n-2+\sigma} {R^{-a} \over 1+ |y|^{2+\alpha}} T^{4+(p-1)\sigma \over n-4} R^{n+1}  \| \psi \|_{*,\alpha , \sigma} \| \bar \phi \|_{in}
\end{align*}
and
\begin{align*}
|B|&\lesssim \left({\la \over d} \right)^{n+2+p \sigma} {R^{(n+1-a)p} \over 1+ |y|^{p(n+1)}}  \| \phi \|_{in} \| \bar \phi \|_{in}\\
&\lesssim \left({\la \over d} \right)^{n-2+\sigma} {R^{-ap} \over 1+ |y|^{2+\alpha}} T^{4+(p-1)\sigma \over n-4} R^{p(n+1)}  \| \phi \|_{in}  \| \bar \phi \|_{in}.
\end{align*}
Using \eqref{defT1} and \eqref{defRT}, we conclude that
$$
\la^{n+2 \over 2}  \left| p \left( (W_2 + \psi + \eta_R \Phi )^{p-1} - W_2^{p-1} \right) \eta_R \bar \Phi \right| \lesssim ( T^\ve + T^\beta ) \| \bar \phi \|_{in}.
$$
We conclude that there exists $b >0$, which depends on $n$ and $\sigma$, so that
$$
|Z(x,t) | \lesssim T^b \| \bar \phi \|_{in}.
$$
This fact gives the validity of \eqref{ll5}.
\end{proof}

\medskip
A last remark in in order.

\medskip

\begin{remark}
Proposition \ref{probesterno} defines the solution to Problem \eqref{sout1} as a function of the initial condition $\psi_0$, in the form of an operator $ \psi = \bar \Psi [\psi_0]$, from a  neighborhood of $0$ in the Banach space $C_0 (\DD)$ equipped with the $C^1$ norm
$\| \psi_0 \|_{L^\infty (\DD) }+ \| \nabla \psi_0 \|_{L^\infty (\DD )}$ into the Banach space of functions $\psi \in L^\infty (\DD)$ equipped with the norm
$ \| \psi \|_{* ,  \alpha , \sigma}$ , defined in \eqref{defstarnorm}.

A closer look to the proof of Proposition \ref{probesterno}, and the Implicit Function Theorem give that $\psi_0 \to \bar \Psi [\psi_0]$ is a diffeomorphism, and that
$$
\| \bar \Psi [\psi_0^1 ] - \bar \Psi [\psi_0^2] \|_{*,\alpha , \sigma} \leq c \left[
\| \psi_0^1 - \psi_0^2\|_{L^\infty (\DD )} +
\| \nabla \psi_0^1 - \nabla \psi_0^2\|_{L^\infty (\DD )} \right],
$$
for some positive constant $c$.
\end{remark}

\section{Finding the parameter functions }\label{parfun}

As mentioned in Section \ref{schema}, we can solve the {\it inner} Problem \eqref{sin2}, provided that certain orthogonality condition
of the "right-hand side" as in \eqref{uffa} are satisfied. In this Section we first derive the system of ordinary differential equations
in $\la_1$ and $d_1$ that is equivalent to get the orthogonality conditions satisfied. Then we find parameter functions  $\la_1 $ and $d_1$
which solve these ODEs. This is done, for any $\phi$ fixed, and satisfying \eqref{bphi}, while $\psi$ is already fixed as the solution
of the {\it outer} problem \eqref{sout}, as stated in Proposition \ref{probesterno}. We conclude the Section showing that the solution $\la_1$ and $d_1$  Lipstitz depends on $\phi$.

\medskip
We start with

\begin{lemma}
  \label{parfunlemma}
  Let $\psi$ be the solution to Problem \eqref{sout}, whose existence and properties are stated in Proposition \ref{probesterno} and Lemma \ref{lemmapsi}. Let $H = H[\la ,  d, \phi , \psi]$ be the function defined in \eqref{defH}.
Let  $\phi$ satisfy the  constraint
\eqref{bphi},
 and parameters $\la_1  $ and $d_1 $ satisfy  the bound \eqref{n1}. We assume $R$ to be a fixed large umber satisfying \eqref{defRT}.
Then, for any $T$ small, we have the validity of the following expansions
\begin{equation}\label{ff1}\begin{split}
\int_{B(0,2R)} H & (y,t) Z_1 (y) \, dy = \la \dot d_1 \left( \int_{\R^n} Z_1^2 \right) \, a_{0,R} \left( 1+ p(t) + \bar q_1 ( {\la_1  \over \la_0 } , {d_1 \over d_0} , \phi )  \right)  \\
&+ \left({\la_0 \over d_0 } \right)^{n-2 +\sigma } \, a_{0,R} \,  \left( 1+  p(t) +  q_1(  {\la_1  \over \la_0 } , {d_1 \over d_0} , \phi ) \right) \\
&+ \la  \dot \la_1 \left({\la_0 \over d_0 } \right)^{n-2} \, R \, a_{0,R}
\left( 1+  p(t) +  q_1(  {\la_1  \over \la_0 } , {d_1 \over d_0} , \phi ) \right)
\end{split}
\end{equation}
and
\begin{equation}\label{ff2}\begin{split}
\int_{B(0,2R)} H & (y,t) Z_0 (y) \, dy = \la \dot \la_1 a_{0,R} A \left(1+ p(t) + \bar q_1 ( {\la_1  \over \la_0 } , {d_1 \over d_0} , \phi) \right)  \\
&- \left({\la_0 \over d_0 } \right)^{n-2} \, B \,  \left[ {\la_1 \over \la_0} - {d_1 \over d_0} \right]   a_{0,R} \, (1+ p(t) + q_1 ( {\la_1  \over \la_0 } , {d_1 \over d_0} , \phi) )\\
 &+ \left({\la_0 \over d_0 } \right)^{n-2 +\sigma  } a_{0,R} \,\left(1 + p(t) +  q_1 ( {\la_1  \over \la_0 } , {d_1 \over d_0} , \phi) \right)\\
 &+   \la \dot d_1  \left({\la_0 \over d_0 } \right)^{n-1}\, R^2 \, a_{0,R}  \, q_1 ( {\la_1  \over \la_0 } , {d_1 \over d_0} , \phi),
\end{split}
\end{equation}
where $A$ and $B$ are the constants given by
\begin{equation}\nonumber  %\label{defA}
A= \int_{\R^n } Z_0^2 (y) \, dy, \quad B={p (n-3) \alpha_n \over 2^{n-2}}  \, \left( \int_{\R^n} U^{p-1}  Z_0 \right).
\end{equation}
Here $\sigma \in (0,1)$ is the number fixed in \eqref{sigmadef}.
With $a_{0,R}$ we denote generic constants (i.e. independent of $t$) with $a_{0,R} =1+  o (R^{-1})$, as $R\to \infty$.
Here $p = p(t)$ denotes a generic function, which is smooth for $t \in (0,T)$ so that, for some $\sigma >0$,  $\| p \|_\sigma$ is uniformly bounded, as $T\to 0$. We refer to \eqref{norma} for the definition of the $\| \cdot \|_\sigma$-norm. The explicit expression of $p=p(t)$ changes from line to line. Moreover, $q_1 = q_1 (\eta_1 , \eta_2 , \phi)$ denotes another generic function, which is smooth in its variable, uniformly bounded, as $t \to T$, for $\eta_1$ , $\eta_2 \in L^\infty (0,T)$, and $\phi $ satisfying \eqref{bphi}, with $q_1 (0,0,0 ) = 0$,
and for any $t \in (0,T)$
\begin{equation}\label{ll7}
\left| q_1 [\eta_1^1 , \eta_2 , \phi] (t) - q_1 [\eta_1^2 , \eta_2 , \phi] (t) \right| \lesssim  \|  \eta_1^1 -\eta_1^2 \|_{L^\infty (0,T)}
\end{equation}
\begin{equation}\label{ll8}
\left| q_1[\eta_1 , \eta_2^1 , \phi] (t) - q_1 [\eta_1 , \eta_2^2 , \phi] (t) \right| \lesssim  \|  \eta_2^1 -\eta_2^2 \|_{L^\infty (0,T)}
\end{equation}
\begin{equation}\label{ll9}
\left| q_1 [\eta_1 , \eta_2 , \phi_1] (t) - q_1 [\eta_1 , \eta_2 , \phi_2] (t) \right| \lesssim  T^\ve \| \phi_1 - \phi_2 \|_{in},
\end{equation}
for some $\ve>0$ small.
The explicit expression of $q_1$ also changes from line to line. The function $\bar q_1$ share the same properties as $q_1$, and moreover
$\bar q_1 (\eta_1 , \eta_2 , \bar \phi + \hat \phi) = \bar q_1 (\eta_1 , \eta_2 , \bar \phi ) + \bar q_1 (\eta_1 , \eta_2 ,  \hat \phi)$.

\end{lemma}

\medskip
\begin{proof}
We write
\begin{equation}\label{hh3}
H= \sum_{j=1}^3 H_j, \quad H_1 =  p \la^{n-2 \over 2} U^{p-1} \psi (\la y + \xi , t (\tau )) , \quad
H_3=  B[ \phi].
\end{equation}
The proof of \eqref{ff1} is consequence of the following three expansions, as $T \to 0$,
\begin{equation}\label{hh4}\begin{split}
&\int_{B(0,2R)} H_2 Z_1 \, dy = \la \dot d_1 \left( \int_{\R^n} Z_1^2 \right) \, a_{0,R} \left( 1+  q_1 ( {\la_1  \over \la_0 } , {d_1 \over d_0} , 0) \right)  \\
&- \left({\la_0 \over d_0 } \right)^{n-1} {p (n-2) \alpha_n \over 2^{n-1}}  \left( \int_{\R^n} U^{p-1} y_1  Z_1 \right) a_{0,R}  \,   \left( 1+ p(t) +  q_1 ( {\la_1  \over \la_0 } , {d_1 \over d_0} , 0) \right) \\
&+ \la  \dot \la_1 \left({\la_0 \over d_0 } \right)^{n-2} \, R \, a_{0,R} \left( 1+  p(t) +  q_1 ( {\la_1  \over \la_0 } , {d_1 \over d_0} , 0) \right),
\end{split}
\end{equation}
\begin{equation}\label{hh5}\begin{split}
\int_{B(0,2R)} H_1 Z_1 \, dy &= \left({\la_0 \over d_0 } \right)^{n-2 +\sigma }  \, a_{0,R} \, \left(1+ p(t) +  q_1 ( {\la_1  \over \la_0 } , {d_1 \over d_0} , \phi ) \right),
\frac{}{}\end{split}
\end{equation}
and
\begin{equation}\label{hh6}\begin{split}
\int_{B(0,2R)} H_3 Z_1 \, dy &= \left({\la_0 \over d_0 } \right)^{2n-2}  \, a_{0,R} \, O( R^{-2} ) \, \bar q_1 ( 0 , 0, \phi)\\
&+ \la_0 \left({\la_0 \over d_0} \right)^{n-2}  \dot d_1  \, O(R^{-1}) \bar q_1 ( 0 , 0 , \phi )
\end{split}
\end{equation}
where we are using the same notations as in the statement of the Lemma.

\medskip
\noindent
{\it Proof of \eqref{hh4}}. \ \ In the region $y \in B(0,2R)$, the function
$$
H_2 = \la^{n+2 \over 2} \left[ e_1 + e_2  - \left({\la_0 \over d_0} \right)^{n-2} \left[ \Delta w + p W_0^{p-1} w \right]
\eta \left( { |x-\xi| \over b \, d_0 } \right) \right] (\la y + \xi , t (\tau ))
$$
has been described in Lemma \ref{lemma2}. Referring to \eqref{app2}, we see immediately that $\int_{B(0,2R)} E_{2,\la } (y,t) Z_1 (y) \, dy =0$, for all $t$, because of symmetry.
We get
\begin{equation*}\begin{split}
\int_{B(0,2R)} & H_2 Z_1 \, dy  = \la \, [ \dot d_1 - {d_0 \over 1+ d} ] \int_{B(0,2R)} Z_1^2 (y) \, dy \\
&+ \la  \dot d_1  \left( {\la_0 \over d_0 } \right)^{n-2}  \int_{B(0,2R)} {\partial h \over \partial y_1} Z_1 (y) \, dy \\
&- \left({\la_0 \over d_0 } \right)^{n-1} {p (n-2) \alpha_n \over 2^{n-1}}  \left( \int_{B(0,2R)} U^{p-1} y_1  Z_1 \right)  \\
&+ \la \dot d_1  \left({\la_0 \over d_0 } \right)^{n-1} O \left( \int_{B(0,2R)} |Z_1 (y) | \, dy \right)\\
& + \la \la_1  \left({\la_0 \over d_0 } \right)^{n-2} O \left( \int_{B(0,2R)} |Z_1 (y) | \, dy \right)\\
& + \left({\la_0 \over d_0 } \right)^{n+2} O \left( \int_{B(0,2R)} |Z_1 (y) | \, dy \right).
\end{split}
\end{equation*}
Expansion \eqref{hh4} follows after we observe that
$$
\int_{B(0,2R)} Z_1^2 (y) \, dy = (\int_{\R^n} Z_1^2 (y) \, dy) (1+ O(R^{2-n} )),
\quad \int_{B(0,2R)} {\partial h \over \partial y_1} Z_1 (y) \, dy = O(R^{-2} ),
$$
$$
\int_{B(0,2R)} U^{p-1} y_1  Z_1  = \left( \int_{\R^n} U^{p-1} y_1  Z_1 \right) (1+ O(R^{-2}) ), \quad
\int_{B(0,2R)} |Z_1 | = O(R),
$$
for $R$ large.

\medskip
\noindent
{\it Proof of \eqref{hh5}}. \ \ From the result of Proposition \ref{probesterno}, and more specifically estimate \eqref{leuco1}-\eqref{leuco11}, we expand
$$
\psi (\xi + \la y , t ) = \psi (\xi ) + \nabla \psi (\bar \xi ) \la y
$$
for some $\bar \xi$. By symmetry, the integral of the first term is zero, so that we get
$$
\int_{B(0,2R)} H_1 Z_1 = ({\la_0 \over d_0} )^{n-2+\sigma } \left( 1+  p(t) +  q_1 ( {\la_1  \over \la_0 } , {d_1 \over d_0} , 0) \right) \, (\int_{B(0,2R)} U^{p-1} y_1 Z_1 ).
$$
Thus we get the validity of \eqref{hh5}.

\medskip
\noindent
{\it Proof of \eqref{hh6}}. \ \ From the definition of the function $B[\phi]$, we immediately observe that this is a linear function of $\phi$, and it does not depend on $\la_1$. A direct computation and the use of the estimate on $\phi$ given in \eqref{bphi} gives
\begin{equation*}\begin{split}
&\int_{B(0,2R)} H_3 Z_1 = \la \dot \la \int_{B(0,2R)}  \left[ {n-2 \over 2} \phi (y, t) + \nabla \phi(y, t) \cdot y \right] Z_1
+ \la \dot d \int_{B(0,2R)} {\partial \phi \over \partial y_1} (y, t) Z_1\\
 &= \left({\la \over d } \right)^{2n-2 +\sigma }  \, a_{0,R} \, O( R^{-1+\alpha} ) \, q_1 ( 0 , 0, \phi)
+ \la \left({\la \over d} \right)^{n-2+\sigma }  \dot d_1  \,  \bar q_1 ( 0 , 0 , \phi ) .
\end{split}
\end{equation*}

\medskip
The proof of \eqref{ff2} is consequence of the following three expansions, as $T \to 0$,
\begin{equation}\label{hh7uu}\begin{split}
&\int_{B(0,2R)} H_2 Z_0 \, dy = \la \dot \la_1 a_{0,R} \left(1+ p(t) + q_1 ( {\la_1  \over \la_0 } , {d_1 \over d_0} , 0) \right) \int_{\R^n } Z_0^2 (y) \, dy \\
&- \left({\la_0 \over d_0 } \right)^{n-2} {p (n-3) \alpha_n \over 2^{n-2}}  \left[ {\la_1 \over \la_0} - {d_1 \over d_0} \right] \left( \int_{\R^n} U^{p-1}  Z_0 \right) \times \\
&\times a_{0,R} \, (1+ p(t) + q_1 ( {\la_1  \over \la_0 } , {d_1 \over d_0} , 0) )\\
&+   \la \dot d_1  \left({\la_0 \over d_0 } \right)^{n-1}\, O (R^2 ) \, q_1 ( {\la_1  \over \la_0 } , {d_1 \over d_0} , 0) + \left({\la_0 \over d_0 } \right)^{n+2} O   (R^{2} ) \, q_1 ( {\la_1  \over \la_0 } , {d_1 \over d_0} , 0) ,
\end{split}
\end{equation}
\begin{equation}\label{hh8}\begin{split}
\int_{B(0,2R)} H_1 Z_0 \, dy &= \left({\la_0 \over d_0 } \right)^{n-2 +\sigma }  \, a_{0,R} \,  \, \left( 1+ p(t) +  q_1 ( {\la_1  \over \la_0 } , {d_1 \over d_0} , \phi ) \right),
\end{split}
\end{equation}
and
\begin{equation}\label{hh9}\begin{split}
\int_{B(0,2R)} H_3 Z_0 \, dy &= \left({\la_0 \over d_0 } \right)^{2n-2+\sigma }  \, a_{0,R} \, O( R^{-1+\alpha} ) \, q_1 ( 0 , 0, \phi)\\
&+ \la_0 \left({\la_0 \over d_0} \right)^{n-2+\sigma}  \dot d_1  \,  \bar q_1 ( 0 , 0 , \phi ),
\end{split}
\end{equation}
where we are using the same notations as in the statement of the Lemma. The proofs of \eqref{hh8} and \eqref{hh9} are similar to the ones of \eqref{hh5} and \eqref{hh6} respectively, so we leave them to the reader.

\medskip
\medskip
\noindent
{\it Proof of \eqref{hh7uu}}. \ \  Referring again to \eqref{app2} for the expression of $H_2$ in the region we are considering,  we see immediately that $\int_{B(0,2R)} E_{2,d } (y,t) Z_0 (y) \, dy =0$, for all $t$, because of symmetry.
We get
\begin{equation*}\begin{split}
\int_{B(0,2R)} & H_2 Z_0 \, dy  =  [ \la \dot \la_1 + \dot \la_0 \la_1 ] \int_{B(0,2R)} Z_0^2 (y) \, dy \\
&+  [ \la \dot \la_1 + \dot \la_0 \la_1 ] \left( {\la_0 \over d_0 } \right)^{n-2}  \int_{B(0,2R)} ({n-2 \over 2} h + \nabla h \cdot y)  Z_0 (y) \, dy \\
&- \left({\la_0 \over d_0 } \right)^{n-2} {p (n-2) \alpha_n \over 2^{n-2}}  \left[ {\la_1 \over \la_0} - {d_1 \over d_0} \right] \left( \int_{B(0,2R)} U^{p-1}  Z_0 \right) \times \\
&\times (1+ p(t) + q_1 ( {\la_1  \over \la_0 } , {d_1 \over d_0} , 0) )\\
&+  \left[ \la \dot d_1  \left({\la_0 \over d_0 } \right)^{n-1} +  \la \dot \la_1  \left({\la_0 \over d_0 } \right)^{n-2}\, \right] O (R^2 ) \, q_1 ( {\la_1  \over \la_0 } , {d_1 \over d_0} , 0) \\
& + \left({\la_0 \over d_0 } \right)^{n+2} O   (R^2 ) \, q_1 ( {\la_1  \over \la_0 } , {d_1 \over d_0} , 0) .
\end{split}
\end{equation*}
Using \eqref{defla0}, that is $\dot \la_0 \int_{\R^n} Z_0^2 = {p \alpha_n \over 2^{n-2}} ({\la_0 \over d_0} )^{n-2} \int_{\R^n} U^{p-1} Z_0 $,
we get
\begin{equation*}\begin{split}
\int_{B(0,2R)} & H_2 Z_0 \, dy  =  \la \dot \la_1 a_{0,R} \left(1+ p(t) + q_1 ( {\la_1  \over \la_0 } , {d_1 \over d_0} , 0) \right) \int_{\R^n } Z_0^2 (y) \, dy \\
&- \left({\la_0 \over d_0 } \right)^{n-2} {p (n-3) \alpha_n \over 2^{n-2}}  \left[ {\la_1 \over \la_0} - {d_1 \over d_0} \right] \left( \int_{\R^n} U^{p-1}  Z_0 \right) \times \\
&\times a_{0,R} \, (1+ p(t) + q_1 ( {\la_1  \over \la_0 } , {d_1 \over d_0} , 0) )\\
&+   \la \dot d_1  \left({\la_0 \over d_0 } \right)^{n-1}\, O (R^2 ) \, q_1 ( {\la_1  \over \la_0 } , {d_1 \over d_0} , 0) + \left({\la_0 \over d_0 } \right)^{n+2} O   (R^{2} ) \, q_1 ( {\la_1  \over \la_0 } , {d_1 \over d_0} , 0) .
\end{split}
\end{equation*}
Thus \eqref{hh7uu} follows.
\end{proof}

\medskip
Next we get the existence and Lipschtz properties of $\la_1$ and $d_1$ that make the required orthogonality conditions.

\medskip
\begin{prop}\label{parfun1}
Let $\psi$ be the solution to Problem \eqref{sout}, whose existence and properties are stated in Proposition \ref{probesterno} and Lemma \ref{lemmapsi}. Let $H = H[\la ,  d, \phi , \psi]$ be the function defined in \eqref{defH}.
For any function $\phi$ satisfying the  constraint
\eqref{bphi},
 there exist functions $\la_1 = \la_1 [\phi] $ and $d_1= d_1 [\phi] $, which satisfy the bound \eqref{n1}, for which
\begin{equation}\label{ma1}\begin{split}
\int_{B(0,2R)} H (y,t) \, Z_0 (y) \, dy& = 0 \quad {\mbox {for all}} \quad t \in (0,T)\\
\int_{B(0,2R)} H(y,t)\, Z_1 (y) \, dy &= 0 , \quad {\mbox {for all}} \quad t \in (0,T).
\end{split}
\end{equation}
Moreover, if $\phi_1$ and $\phi_2$ satisfy \eqref{bphi}, one has
\begin{equation}\label{ll10}\begin{split}
\| \dot \la_1 [\phi_1] - \dot \la_1 [\phi_2] \|_{1+\sigma \over n-4} +
\| \dot d_1 [\phi_1] - \dot d_1 [\phi_2] \|_{1+\sigma\over n-4} \leq T^b \| \phi_1 -\phi_2 \|_{in}
\end{split}
\end{equation}
for some fixed number $b >0$, which depends on $n$ and $\sigma$.

\end{prop}

\begin{proof} The result of Lemma \ref{parfunlemma} is telling us that solving equation \eqref{ma1} is equivalent to solving a certain non linear
non local system of ordinary differential equation of first order in $\la_1$ and $d_1$.
Indeed, from \eqref{ff1}   we get that the second equation in \eqref{ma1} is equivalent to
\begin{equation}\label{ma0}\begin{split}
\dot d_1 + A_R (T-t)^{2 \over n-4} &= (T-t)^{1+\sigma \over n-4} \left( p(t) +
q_1 \left( {\la_1 \over \la_0} , {d_1 \over d_0} , \phi \right) \right)\\
&+ (T-t)^{1+ {1\over n-4}} \dot \la_1 \left( 1+ p(t) +
q_1 \left( {\la_1 \over \la_0} , {d_1 \over d_0} , \phi \right) \right)
\end{split}
\end{equation}
where $A_R$ is a constant (independent of $t$), with the property that $A_R \sim A_\infty (1+ O(R^{-1} ))$, as $R\to \infty$, with a fixed positive constant $A_\infty$.
The functions $p$ and $q_1$ have the same properties as stated in Lemma \ref{parfunlemma}.

\medskip
We next look at \eqref{ff2} to get the differential equation corresponding to the first equation in \eqref{ma1}. Using \eqref{def2}-\eqref{def3}-\eqref{defla0}, we get
\begin{equation*} \begin{split}
\dot \la_1 - {(n-3 ) \over (T-t)} \la_1 &=  (T-t)^{1+\sigma \over n-4} f(t) + {\la_1 \over (T-t) } q_1 \left( {\la_1 \over \la_0} , {d_1 \over d_0} , \phi \right) \\
&
+ (T-t)^{1+ {1\over n-4}} \dot d_1 \left( 1+ p(t) +
q_1 \left( {\la_1 \over \la_0} , {d_1 \over d_0} , \phi \right) \right)
\end{split}
\end{equation*}
Here $f=f(t)$ stands for a uniformly bounded in $(0,T)$, as $T \to 0$. It is convenient to multiply the above equation against $(T-t)^{n-3}$, and re-write it as
\begin{equation}\label{ma2} \begin{split}
 {d \over dt} \left( (T-t)^{n-3} \la_1 \right) &=  (T-t)^{n-3+{1+\sigma \over n-4}} f(t) +  (T-t)^{n-4}\, \la_1 \, q_1 \left( {\la_1 \over \la_0} , {d_1 \over d_0} , \phi \right) \\
&
+ (T-t)^{n-2+ {1\over n-4}} \dot d_1 \left( 1+ p(t) +
q_1 \left( {\la_1 \over \la_0} , {d_1 \over d_0} , \phi \right) \right)
\end{split}
\end{equation}
\medskip
Let ${\bf d}= {\bf d}(t) $ and $\Lambda = \Lambda (t)$ be the solution to
\begin{equation} \nonumber  \begin{split}
\dot {\bf d} +  A_R (T-t)^{2 \over n-4} &= (T-t)^{1+\sigma\over n-4} p(t)  \\
 {d \over dt} \left( (T-t)^{n-3} \Lambda \right) &=  (T-t)^{n-3+{1+ \sigma \over n-4}} f(t) ,
\end{split}
\end{equation}
given by
\begin{equation}\label{ma4} \begin{split}
{\bf d} (t) &= \int_t^T  \left[- A_R (T-s)^{1+ \sigma\over n-4} + (T-s)^{1+ \sigma\over n-4} p(s) \right]\, ds  \\
\Lambda (t) &=  {1\over (T-t)^{n-3}} \int_t^T (T-s)^{n-3+{1+ \sigma\over n-4}} f(s) \, ds.
\end{split}
\end{equation}
These functions satisfy the bound \eqref{n1}.
Then $d_1 = {\bf d} + d$, $\la_1 = \Lambda + \la$ solves the system \eqref{ma0}-\eqref{ma2} if
\begin{equation}\nonumber \begin{split}
d (t) &=\int_t^T  (T-s)^{1+ \sigma \over n-4} \,
q_1 \left( {\Lambda +\la \over \la_0} , {{\bf d} + d \over d_0} , \phi \right) (s) \, ds\\
&+ \int_t^T (T-s)^{1+ {1\over n-4}} ( \dot \Lambda + \dot \la ) (s) \left( 1+ p(s) +
q_1 \left( {\Lambda +\la \over \la_0} , {{\bf d} + d  \over d_0} , \phi \right) \right)\, ds \\
\la (t) &=  {1\over (T-t)^{n-3}} \int_t^T (T-s)^{n-4}\, \la_1 (s) \, q_1 \left( {\Lambda + \la \over \la_0} , {{\bf d} + d \over d_0} , \phi \right)\, ds \\
&
+ {1\over (T-t)^{n-3}} \int_t^T (T-s)^{n-2+ {1\over n-4}} \dot d_1 (s) \left( 1+ p(s) +
q_1 \left( {\Lambda +\la \over \la_0} , {{\bf d} + d \over d_0} , \phi \right) \right).
\end{split}
\end{equation}
Using again the result of Lemma \ref{parfun}, and in particular \eqref{ll7}-\eqref{ll8}, one can  solve \eqref{ma4} with a fixed point argument based on the Contraction Mapping Theorem.
Estimate \eqref{ll10} follows from \eqref{ma4} and \eqref{ll9}.

\end{proof}

\section{Solving the inner problem}\label{inner}

The last step in the proof of our result is to solve  the {\it inner} Problem \eqref{sin2},  after we already defined the {\it outer } solution $\psi$, whose existence and properties are contained in Proposition \ref{probesterno} and Lemma \ref{lemmapsi} in Section \ref{parfun}, and the parameter functions $d_1$, $\la_1$,  as in Proposition \ref{parfun1} in Section \ref{parfun}.

\medskip
The key ingredient to solve \eqref{sin2} for functions $\phi$ satisfying \eqref{bphi} is the resolution of following linear problem: Given a sufficiently large number $R>0$,
  construct a solution $(\phi , e_0) $ to the initial value problem
\be \label{p110}
\phi_\tau  =
\Delta \phi + pU(y)^{p-1} \phi + h(y,\tau )  \inn B_{2R} \times (\tau_0, \infty )
\ee
$$
\phi(y,\tau_0) = e_0Z (y)  \inn B_{2R}
$$
provided that $h$ satisfies certain time-space decay rate and certain orthogonality conditions.
Here $Z$ is the positive radially symmetric bounded eigenfunction
associated to the only negative eigenvalue to the linear problem \eqref{eigen0}.
We recall that $\tau = \tau (t)$ is given in \eqref{deftau}, as
$$
\tau (t) \sim  {n-4 \over (n-2) \ell}  (T-t)^{-1-{2\over n-4}}  \quad {\mbox {as}} \quad t \to T, \quad {\mbox {and}} \quad \tau_0 = \tau (0).
$$
In the $\tau$-variable, the bound \eqref{bphi} on $\phi$ reads as
\begin{equation}\nonumber %\label{bbphi}
\begin{split}
\| \phi \|_{in} & = \sup_{\tau > \tau_0, \, y \in B(0,2R)}  \tau^\nu \, R^{-n-1+a} \, (1+ |y|^{n+1}) |\phi (y,\tau)| \\
&+\sup_{\tau > \tau_0, \, y \in B(0,2R)}  \tau^\nu \, R^{-n-1+a} \, (1+ |y|^{n+2}) |\nabla \phi (y,\tau)| \lesssim  R^{-1}
\end{split}
\end{equation}
where $\nu $ is given by
$$
\tau^{-\nu} \sim \left({\la \over d} \right)^{n-2+\sigma }  \sim (T-t)^{n-2 +\sigma \over n-4}, \quad {\mbox {as}} \quad t\to T.
$$
Here  $\sigma \in ({1\over 2} , 1)$  is the constant (which can be thought close to $1$) introduced in \eqref{n1}.
The solution for Problem \eqref{p110} we build has $R$-dependent uniform bounds for right hand-side $h$ with $L^\infty$-weighted norms of the type
$$
\|h\|_{\nu, 2+a } := \sup_{\tau > \tau_0 }  \sup _{y\in B_{2R}}   \tau^{\nu} \, \la^2 \, (1+ |y|^{2+a} ) \, |h( y ,\tau ) | .
$$
Also, for a function $p=p(\tau)$ we denote
$$
\|p\|_{\nu}:= \sup_{\tau > \tau_0 } \tau^{\nu} |p(\tau ) | .
$$

\medskip
We have the validity of

\begin{prop} \label{prop0}
Let $R>0$ be large enough. For any  $\tau_0 $ sufficiently large (depending on $R$), for any  $h=h(y,\tau)$  with  $\|h\|_{\nu, 2+a} <+\infty$
that satisfies for all $j=0,1$
\be
 \int_{B (0,2R)} h(y ,\tau)\, Z_j (y) \, dy\ =\ 0  \foral \tau\in (\tau_0, \infty),
\label{ortio}\ee
there exist  $\phi = \phi [h]$  and $e_0 = e_0 [h]$ which solve Problem $\equ{p110}$. They define linear operators of $h$
that satisfy the estimates
\be
 (1+ |y| )  | \nabla \phi (y,\tau ) | +  |\phi(y,\tau) |  \ \lesssim   \  \tau^{-\nu} \,
  \frac {R^{n+1-a}}{1+ |y|^{n+1}} \|h \|_{\nu, 2+a} ,
\label{cta1}\ee
and
\be
 | e_0[h]| \, \lesssim \,  \|h\|_{ \nu, 2+a}.
\label{cot2n}\ee
%where $$\bar h :=  h -  \left(\int_{B_{2R}} hZ_0\right) \, Z_0. $$
\end{prop}

%\medskip
%Observe that
%$$
%| \phi(y,t) | \lesssim \| \phi\|_a {1\over 1+ |y|^a } \, (T-t )^{{n-2 +\sigma \over n-4} }
%$$
%for $|y| \geq 2R$.

\medskip
The proof of this Proposition is a straightforward
 adaptation to our symmetric setting of the result contained in Proposition 7.1 in \cite{CDM}. We thus refer the reader to \cite{CDM}
 for the proof of this result.

\medskip
Proposition \ref{prop0} states the existence of a linear operator $\Xi$ which to any function $h(y, \tau )$, with $\| h \|_{\nu , 2+a} $-bounded and satisfying \eqref{ortio}, associates the solution $(\phi , e_0)$ to \eqref{p110}. Furthermore, it states that $\Xi$ is continuous between $L^\infty$ spaces equipped with the topologies described by \eqref{cta1}-\eqref{cot2n}.

\medskip
We want to use Proposition \ref{prop0} to solve the {\it inner } problem \eqref{sin2}. Up to this moment in our argument, the radius $R$ is chosen large and the final time $T$ was chosen small, with $R^{p (n+1) } T^{4\over n-4} <1$. See \eqref{defT1} and \eqref{defRT}.
We claim that Problem \eqref{sin2} has a solution.
 Indeed, we observe first that the parameter functions $\la_1$ and $d_1$ as defined in Section \ref{parfun} are such that the right-hand side $H (y,t)$ satisfies the orthogonality condition \eqref{ortio}, for any $t \in (0,T)$ (or equivalently for any $\tau > \tau_0$).
Thus,
the existence and properties of $\phi$ and $e_0$ solution to \eqref{sin2}  are reduced to find a fixed point for
$$
\phi =  {\mathcal A} (\phi ) , \quad {\mbox {where}} \quad {\mathcal A} (\phi) :=  \Xi  \left( H ( \la_1 [\phi]  , d_1 [\phi] , \phi , \psi [\phi] ) \right)
$$
in a proper set of functions. We recall the definition of $H$ given in  \eqref{defH}
\begin{equation*}
\begin{split}
H[\la , d , \phi, \psi] (y,\tau ) &= p \la^{n-2 \over 2} U^{p-1} \psi (\la y + \xi , t (\tau )) \\
&+\la^{n+2 \over 2} E_2 (\la y + \xi , t (\tau )) + B[ \phi],
\end{split}
\end{equation*}
where we recall that $E_2=  e_1 + e_2  - \left({\la_0 \over d_0} \right)^{n-2} \left[ \Delta w + p W_0^{p-1} w \right]$.
From Lemma \ref{lemma2}, we get that
$$
\| \la^{n+2 \over 2} \,E_2 (\la y + \xi , t (\tau )) \|_{\nu , 2+a } \leq C \, T^{1-\sigma \over n-4},
$$
while from Proposition \ref{probesterno}
we have the existence of  $\bar \ve >0$ so that
$$
\| p \la^{n-2 \over 2} U^{p-1} \psi (\la y + \xi , t (\tau )) \|_{\nu, 2+a} \leq C \,  T^{\bar \ve}.
$$
Here $C$ is a positive fixed constant.
These estimates suggest to search for a fixed point for the map ${\mathcal A} $ in the set of functions $\phi$ so that
$$
{\mathcal C} := \{ \phi \, : \, \| \phi \|_{in} \leq \, r \,  T^\beta \}, \quad \beta = \max \{ {1-\sigma \over n-4} , \bar \ve \}
$$
for some $r$ large, independent of $T$ and $R$.
From \eqref{defB} and our assumptions \eqref{defRT} on $R$, we easily get
\begin{equation}\nonumber  %\label{ill}
\| B[\phi] \|_{\nu , 2+a} \leq C \,  T^{1+{1\over n-4}} R^{n+1 -a}  \| \phi \|_{in} \leq C \,  \| \phi \|_{in},
\end{equation}
for some positive constant $C$.
This implies that, provided the constant $r$ is chosen large, one has ${\mathcal A} (C) \subset C$.
 We next prove that ${\mathcal A}$ is a contraction mapping, provided $R$ is (possibly) larger (and thus $T$ smaller). We shall emphasize the fact that $\psi $ depends from $\phi$ in a non linear and non local way, recalling that
$$
\psi = \psi [\phi] = \psi [ \la_1 (\phi) , d_1 (\phi) , \phi ].
$$
Combining \eqref{ll3}, \eqref{ll4}, \eqref{ll5} and \eqref{ll10}, one gets
\begin{equation}\label{gia3}
\| \psi [\phi_1 ] - \psi [\phi_2 ] \|_{*,\alpha, \sigma} \leq {\bf c} \| \phi_1 - \phi_2 \|_{in}
\end{equation}
for some ${\bf c} \in (0,1)$, which can be done arbitrarily small, provided $T$ is chosen small enough.
We claim that there exists ${\bf c} \in (0,1)$ so that
\begin{equation}\nonumber %\label{gia2}
\| {\mathcal A } (\phi_1 ) - {\mathcal A} (\phi_2 ) \|_{in} \leq {\bf c} \| \phi_1 - \phi_2 \|_{in}
\end{equation}
for any $\phi_1$, $\phi_2 \in {\mathcal C}$. From Proposition \ref{prop0} we get that
\begin{equation}\label{gia000}
\| {\mathcal A } (\phi_1 ) - {\mathcal A} (\phi_2 ) \|_{in} \leq c\, R^{n+1-a}  \| H[\la_1 , d_1 , \phi_1 , \psi_1] - H[\la_2 , d_2 , \phi_2, \psi_2] \|_{\nu , 2+a}
\end{equation}
where $\la_i = \la [\phi_i]$, $d_i = d [\phi_i] $ and $\psi_i = \psi [\phi_i]$.
Consider first $$H_1 [\phi] := p \la^{n-2 \over 2} [\phi]  U^{p-1} \psi [\phi] (\la [\phi] y + \xi [\phi]  , t (\tau )).$$
We write
\begin{equation*} \begin{split}
\left| H_1 [\phi_1] - H_1 [\phi_2] \right|&\lesssim  p \la_1^{n-2 \over 2} U^{p-1} \left|\psi [\phi_1] (\la_1 y + \xi_1  , t (\tau )) -  \psi [\phi_2] (\la_1 y + \xi_1   , t (\tau )) \right|\\
&+ p \left( \la_1^{n-2 \over 2} -\la_2^{n-2 \over 2} \right)  U^{p-1} \left|\psi [\phi_2] (\la_1 y + \xi_1  , t (\tau ))  \right| \\
&+ p \la_2^{n-2 \over 2} U^{p-1} \left| \psi [\phi_2] (\la_1 y + \xi_1   , t (\tau )) - \psi [\phi_2] (\la_2 y + \xi_2 , t (\tau ))  \right|\\
&= h_1 + h_2 +h_3 .
\end{split}
\end{equation*}
Thanks to \eqref{ll5}-\eqref{gia3}, there exists a number $b>0$, which depends on $n$ and $\sigma$, for which
$$
|h_1 (y,\tau ) | \leq c\, T^b \,  {\tau^{-\nu } \over 1+ |y|^4} \| \phi_1 - \phi_2 \|_{in},
$$
for some constants $c>0$. Also, thanks to \eqref{ll10}, we have
\begin{equation*}\begin{split}
|(h_2 + h_3) (y,\tau ) | &\leq c  \, {\tau^{-\nu } \over 1+ |y|^4} \left( \|\dot \la [\phi_1] - \dot \la [\phi_2] \|_{{1+ \sigma \over n-4}} + \|\dot d [\phi_1] - \dot d [phi_2] \|_{{1+\sigma \over n-4}} \right) \\
&\leq c \, T^{b} {\tau^{-\nu } \over 1+ |y|^4} \| \phi_1 - \phi_2 \|_{in},
\end{split}\end{equation*}
for some fixed constant $b>0$.
We conclude that, for some constants $c$ and $b>0$,
\begin{equation}\label{gia4}
 \| H_1 [\phi_1] - H_1 [\phi_2] \|_{\nu, 2+a} \leq T^b \| \phi_1 -\phi_2 \|_{in}
\end{equation}
 for some ${\bf c} \in (0,1)$.
 Next, we consider
 $
 E_2 [\phi] (\la [\phi]  y + \xi[\phi] , t (\tau ) ).
 $ Since the part of $E_2$ given by $- \left({\la_0 \over d_0} \right)^{n-2} \left[ \Delta w + p W_0^{p-1} w \right]$ does not depend on $\phi$,  we have, for $\la_i = \la[\phi_i]$, $d_i= d[\phi_i]$
 \begin{align*}
 \la_1^{n+2 \over 2} &E_2 [\phi_1] (\la_1 y + \xi_1 , t (\tau ) ) -  \la_2^{n+2 \over 2} E_2 [\phi_2] (\la_2 y + \xi_2 , t (\tau ) ) \\
 &=  \left[  \left( \la_1 \dot d_1 +{1\over (\la_1 y + \xi_1 )_1} \right) - \left( \la_2 \dot d_2 +{1\over (\la_2 y + \xi_2 )_1} \right) \right] Z_1 {y} \\
 &+ \left( \la_1 \dot \la_1 - \la_2  \dot \la_2 \right) Z_0 (y) - p U^{p-1} (y) \left( U(y+ {2d_1 \over \la_1} ) - U(y+{2d_2 \over \la_2} ) \right).
 \end{align*}
 Using the fact that we are in the region $|y|<2R$ and the validity of \eqref{ll10} we get
 \begin{equation}\label{eric}
 \| \la_0^{n+2 \over 2} \left[ E_2 [\phi_1] - E_2 [\phi_2] \right] \|_{\nu, 2+a} \leq c T^{b} \| \phi_1 -\phi_2 \|_{in},
 \end{equation}
 for some constants $c$ and $b>0$. Moreover we have that
 \begin{align*}
 B[\phi_1 ] - B[\phi_2] & = \la_1 \dot \la_1 \left[ {n-2 \over 2} ( \phi_1 - \phi_2)  + \nabla (\phi_1 -\phi_2) \cdot y \right] \\
 & +(\la_1 \dot \la_1 - \la_2 \dot \la_2) \left[ {n-2 \over 2} \phi_2 + \nabla \phi_2 \cdot y \right]\\
 & + \left[ \la_1 \dot d_1 + {\la_1 \over (\la_1 y + \xi_1 )_1 } \right] {\partial (\phi_1 -\phi_2) \over \partial
 y_1}  \\
 &+ \left( \left[ \la_1 \dot d_1 + {\la_1 \over (\la_1 y + \xi_1 )_1 } \right] - \left[ \la_2 \dot d_2 + {\la_2 \over (\la_2 y + \xi_2 )_1 } \right] \right) {\partial \phi_2 \over \partial
 y_1} .
\end{align*}
Using again \eqref{ll10}, we get
\begin{equation}\label{gia5}
 \| B[\phi_1 ] - B[\phi_2]   \|_{\nu, 2+a} \leq c T^{b} \| \phi_1 -\phi_2 \|_{in},
 \end{equation}
 for some constants $c$ and $b>0$. Inserting estimates \eqref{gia4}-\eqref{eric}-\eqref{gia5} into \eqref{gia000}, we conclude that
$$
 \| {\mathcal A } (\phi_1 ) - {\mathcal A} (\phi_2 ) \|_{in} \leq c\, R^{n+1-a}  T^{b} \| \phi_1 -\phi_2 \|_{in}.
$$
Choosing, if necessary, $T$ even smaller, we get
\begin{equation}\label{gia6}
 \| {\mathcal A } (\phi_1 ) - {\mathcal A} (\phi_2 ) \|_{in} \leq {\bf c} \| \phi_1 -\phi_2 \|_{in}
\end{equation}
for some ${\bf c} \in (0,1)$.

\medskip
Estimate \eqref{gia6} gives the contraction property for ${\mathcal A}$ in the set ${\mathcal C}$.
 Thus we proved the existence of a solution to the {\it inner } problem \eqref{sin2}. This fact concludes the proof of the existence of the solution predicted by Theorem \ref{teo}, with the expected properties.

\medskip
\noindent
{\bf Acknowledgements:}
M. del Pino has been supported by a UK Royal Society
Research Professorship and Grant PAI AFB-170001, Chile. M. Musso has been
partly supported by EPSRC grant EP/T008458/1, UK. The research of J. Wei is
partially supported by NSERC of Canada.

\end{document}